\documentclass[11pt]{amsart}
\usepackage{geometry}                
\usepackage{graphicx}
\usepackage{amssymb}
\usepackage{epstopdf}
\usepackage{geometry} 
\usepackage{graphicx}
\usepackage{amsmath} 
\numberwithin{equation}{section}
\usepackage{commath}
\usepackage{enumitem}
\usepackage{changepage}
\usepackage{cleveref}
\usepackage[normalem]{ulem}
\usepackage{bbm}
\usepackage{subfig}
\usepackage{float}
\usepackage{url}
\usepackage{mathtools}
\usepackage{amsthm} 
\usepackage[disable]{todonotes} 
\DeclareMathOperator{\sgn}{sgn}
\DeclareMathOperator*{\esssup}{ess\,sup}
\newcommand\wanteq{\stackrel{\mathclap{\normalfont\mbox{defined}}}{=}}
\theoremstyle{plain}
\newtheorem{theorem}{Theorem}[section]
\newtheorem*{theorem*}{Theorem}
\newtheorem{lemma}[theorem]{Lemma}

\newtheorem{proposition}[theorem]{Proposition}
\newenvironment{claim}[1]{\par\noindent\underline{Claim.}\space#1}{}
\newenvironment{claimproof}[1]{\par\noindent\emph{Proof of Claim.}\space#1}{\leavevmode\unskip\penalty9999 \hbox{}\nobreak\hfill\quad\hbox{$\blacksquare$}}
\theoremstyle{remark}
\newtheorem*{remark}{Remark}
\theoremstyle{definition}
\newtheorem{definition}[theorem]{Definition}
\allowdisplaybreaks
\title[Stability and Uniqueness for Burgers--Hilbert]{Stability and uniqueness for piecewise smooth solutions to Burgers--Hilbert among a large class of solutions}
\author[Krupa]{Sam G. Krupa}
\address{Department of Mathematics\\ The University of Texas at Austin\\ Austin, TX 78712\\ USA}
\email{skrupa@math.utexas.edu}
\author[Vasseur]{Alexis F. Vasseur}
\address{Department of Mathematics\\ The University of Texas at Austin\\ Austin, TX 78712\\ USA}
\email{vasseur@math.utexas.edu}
\thanks{The second author was partially supported by NSF Grant DMS-1614918.}
\date{April 20th, 2019}                                           

\begin{document}
\begin{abstract}
In this paper, we show uniqueness and stability for the piecewise-smooth solutions to the Burgers--Hilbert equation constructed in Bressan and Zhang [{\em Commun. Math. Sci.}, 15(1):165--184, 2017]. The Burgers--Hilbert equation is $u_t+(\frac{u^2}{2})_x=\mathbf{H}[u]$ where $\mathbf{H}$ is the Hilbert transform, a nonlocal operator. We show stability and uniqueness for solutions amongst a larger class than the uniqueness result in Bressan and Zhang. The solutions we consider are measurable and bounded, satisfy at least one entropy condition, and verify a strong trace condition. We do not have smallness assumptions . We use the relative entropy method and theory of shifts (see Vasseur [{\em Handbook of Differential Equations: Evolutionary Equations}, 4:323 -- 376, 2008]).
\end{abstract}
\maketitle
\tableofcontents

\section{Introduction}

We consider a scalar balance law in one space dimension with a source term,
\begin{align}
\label{system}
\begin{cases}
\partial_t u + \partial_x A(u)=G(u(\cdot,t))(x),\mbox{ for } x\in\mathbb{R},\mbox{ } t>0,\\
u(x,0)=u^0(x)
\end{cases}
\end{align}

The unknown is $u:\mathbb{R}\times[0,T)\to \mathbb{R}$. The function $A:\mathbb{R}\to\mathbb{R}$ is the \emph{flux function}. In this paper, we only consider $A\in C^3(\mathbb{R})$ which are strictly convex. The function $G:L^2(\mathbb{R})\to L^2(\mathbb{R})$ is Lipschitz continuous and translation invariant. The \emph{initial data} is $u^0:\mathbb{R}\to\mathbb{R}$.

We consider both bounded  \emph{classical} and bounded \emph{weak} solutions to \eqref{system}. A weak solution $u$ is bounded and measurable and satisfies \eqref{system} in the sense of distributions. I.e., for every Lipschitz continuous test function $\Phi:\mathbb{R}\times[0,T)\to \mathbb{R}$ with compact support,
\begin{equation}
\begin{aligned}\label{u_solves_equation_integral_formulation_chitchat}
\int\limits_{0}^{T} \int\limits_{-\infty}^{\infty} \Bigg[\partial_t\Phi u + \partial_x\Phi f(u) \Bigg]\,dxdt +\int\limits_{-\infty}^{\infty} \Phi(x,0)u^0(x)\,dx
\\
=-\int\limits_{0}^{T} \int\limits_{-\infty}^{\infty}\Phi G(u(\cdot,t))(x)\,dxdt.
\end{aligned}
\end{equation}

A pair of functions $\eta,q\colon\mathbb{R}\to\mathbb{R}$ are called an entropy, and entropy-flux pair, respectively, for \eqref{system}  if they satisfy the compatibility relation $q'=A'\eta'$. We only consider solutions $u$ which are entropic for the entropy $\eta$. That is, they satisfy the following entropy condition:
\begin{align}\label{entropy_condition_distributional_system_chitchat}
\partial_t \eta(u)+\partial_x q(u) \leq \eta'(u)G(u(\cdot,t))(x),
\end{align}
in the sense of distributions. I.e., for all positive, Lipschitz continuous test functions $\phi:\mathbb{R}\times[0,T)\to\mathbb{R}$ with compact support:
 \begin{equation}
\begin{aligned}\label{u_entropy_integral_formulation_chitchat}
\int\limits_{0}^{T} \int\limits_{-\infty}^{\infty}\Bigg[\partial_t\phi\big(\eta(u(x,t))\big)+&\partial_x \phi \big(q(u(x,t))\big)\Bigg]\,dxdt+ \int\limits_{-\infty}^{\infty}\phi(x,0)\eta(u^0(x))\,dx\geq
\\
&-\int\limits_{0}^{T} \int\limits_{-\infty}^{\infty}\phi\eta'(u(x,t))G(u(\cdot,t))(x)\,dxdt.
\end{aligned}
\end{equation}

We study solutions $u$ to \eqref{system} among the class of functions verifying a strong trace property (first introduced in \cite{Leger2011}):

\begin{definition}\label{strong_trace_definition}
Fix $T>0$. Let $u\colon\mathbb{R}\times[0,T)\to\mathbb{R}$ verify $u\in L^\infty(\mathbb{R}\times[0,T))$. We say $u$ has the \emph{strong trace property} if for every fixed Lipschitz continuous map $h\colon [0,T)\to\mathbb{R}$, there exists $u_+,u_-\colon[0,T)\to\mathbb{R}$ such that
\begin{align}
\lim_{n\to\infty}\int\limits_0^{t_0}\esssup_{y\in(0,\frac{1}{n})}\abs{u(h(t)+y,t)-u_+(t)}\,dt=\lim_{n\to\infty}\int\limits_0^{t_0}\esssup_{y\in(-\frac{1}{n},0)}\abs{u(h(t)+y,t)-u_-(t)}\,dt=0
\end{align}
for all $t_0\in(0,T)$.
\end{definition}

Note that for example a function $u\in L^\infty(\mathbb{R}\times[0,T))$ will satisfy the strong trace property if for each fixed $h$, the right and left limits
\begin{align}
\lim_{y\to0^{+}}u(h(t)+y,t)\hspace{.7in}\mbox{and}\hspace{.7in}\lim_{y\to0^{-}}u(h(t)+y,t)
\end{align}
exist for almost every $t$. In particular, a function $u\in L^\infty(\mathbb{R}\times[0,T))$ will have strong traces according to \Cref{strong_trace_definition} if $u$ has a representative which is in $BV_{\text{loc}}$. However, the strong trace property is weaker than $BV_{\text{loc}}$.

Our method is the relative entropy method, a technique created by Dafermos \cite{doi:10.1080/01495737908962394,MR546634} and DiPerna \cite{MR523630}  to give $L^2$-type stability estimates between a Lipschitz continuous solution and a rougher solution, which is only weak and entropic for a strictly convex entropy (the so-called \emph{weak-strong} stability theory). For a system \eqref{system} endowed with an entropy $\eta$, the technique of relative entropy considers the quantity called the
\emph{relative entropy}, defined as
\begin{align}
\eta(u|v)\coloneqq \eta(u)-\eta(v)-\eta'(v)\cdot (u-v),
\end{align}
for $u,v\in\mathbb{R}$.

Similarly, we define relative entropy-flux,
\begin{align}
q(u;v)\coloneqq q(u)-q(v)-\eta'(v)\cdot (f(u)-f(v)).
\end{align}

Remark that for any constant $v\in\mathbb{R}$, the map $u\mapsto\eta(u|v)$ is an entropy for the system \eqref{system}, with associated entropy flux $u\mapsto q(u;v)$. Furthermore, if $u$ is a weak solution to \eqref{system} and entropic for $\eta$, then for a fixed $v\in\mathbb{R}$, $u$ will also be entropic for $\eta(\cdot|v)$. This can be calculated directly from \eqref{system} and \eqref{entropy_condition_distributional_system_chitchat} -- note that the map $u\mapsto\eta(u|v)$ is basically $\eta$ plus a linear term.

Moreover, by virtue of $\eta$ being \emph{strictly} convex, the relative entropy is comparable to the $L^2$ distance, in the following sense:

Due to the strict convexity of $\eta\in C^2(\mathbb{R})$, for $a,b\in\mathbb{R}$ in a fixed compact set, there exists $c*,c^{**}>0$ such that
\begin{align}\label{entropy_relative_L2_control_system}
c^*(a-b)^2\leq\eta(a|b)\leq c^{**}(a-b)^2.
\end{align}
The constants $c^*,c^{**}$ depend on the fixed compact set and bounds on the second derivative of $\eta$.

For a Lipschitz solution $\bar{u}$ to \eqref{system}, and a weak, entropic solution $u$, the method of relative entropy gives estimates on the growth in time of the quantity
\begin{align}
\norm{\bar{u}(\cdot,t)-u(\cdot,t)}_{L^2(\mathbb{R})}
\end{align}
by considering the time derivative $\partial_t\int\eta(u|\bar{u})\,dx$ and using that \eqref{entropy_relative_L2_control_system} gives the $L^2$ stability.

When a discontinuity is introduced into the otherwise smooth $\bar{u}$, the method of relative entropy breaks down. In fact, simple examples for the scalar conservation laws show that we cannot hope to get stability between $\bar{u}$ and $u$ in the form of the classical weak-strong estimates when $\bar{u}$ has a discontinuity. 

However, we can get weak-strong stability results via the relative entropy method when we allow the discontinuity in $\bar{u}$ to move with a speed which is artificially dictated and depends on $u$. This is the theory of stability up to a shift -- a program started in \cite{VASSEUR2008323} for combining shifts with the method of relative entropy. The theory of stability up to a shift has been undergoing heavy development through the last decade. The first result was for pure shock wave initial data for the scalar conservation laws \cite{Leger2011_original}. Then, by introducing the technique of a-contraction, progress has been made on pure shock wave initial data for systems \cite{MR3519973,MR3479527,MR3537479,serre_vasseur,Leger2011}. Furthermore, there are results for scalar viscous conservation laws in both one space dimension \cite{MR3592682} and multiple \cite{multi_d_scalar_viscous_9122017}. Recent work for scalar has allowed for many discontinuities to exist in the otherwise smooth $\bar{u}$ and shift each discontinuity independently as needed to keep $L^2$ stability, and in this way the method of relative entropy has in fact been used to show uniqueness for an arbitrary weak entropic solution (entropic for a single entropy) amongst the class of all weak solutions entropic for the same entropy. See \cite{2017arXiv170905610K}. For a general overview of theory of shifts and the relative entropy method, see \cite[Section 3-5]{MR3475284}. 
The theory of stability up to a shift has also been used to study the asymptotic limits when the limit is discontinuous (see \cite{MR3333670} for the scalar case, \cite{MR3421617} for systems). There are many other results using the relative entropy method to study the asymptotic limit. However, without the theory of shifts these results can only consider limits which are Lipschitz continuous (see  \cite{MR1842343,MR1980855,MR2505730,MR1115587,MR1121850,MR1213991,MR2178222,MR2025302} and \cite{VASSEUR2008323} for a survey).
 
The present paper is another step in this program of stability up to a shift. 

A case of \eqref{system} of particular interest is the Burgers--Hilbert equation,
\begin{align}\label{burgert_hilbert_system}
u_t+\Bigg(\frac{u^2}{2}\Bigg)_x=\mathbf{H}[u],\\
u(x,0)=u^0(x),
\end{align}
where $\mathbf{H}[u]$  is the Hilbert transform of $u$ and $u^0:\mathbb{R}\to\mathbb{R}$ is the initial data. The Hilbert transform of a function $f\in L^2(\mathbb{R})$ is defined as 
\begin{align}
\widehat{\mathbf{H}[f]}(\xi)\coloneqq-i\frac{\xi}{\abs{\xi}}\widehat{f}(\xi).
\end{align}
Note that the Hilbert transform is an isometry on $L^2(\mathbb{R})$, and hence Lipschitz of course. Note also the Hilbert transform is nonlocal. For initial data $u^0\in H^2(\mathbb{R})$, local existence and uniqueness for \eqref{burgert_hilbert_system} is given in \cite{MR2982741}, along with a precise estimate on how long the solution remains regular (for a shorter proof, see \cite{MR3348783}).  For general initial data $u^0\in L^2(\mathbb{R})$, the global existence in time of weak, entropic solutions to \eqref{burgert_hilbert_system} is shown in \cite{MR3248030}. In \cite{MR3248030},  the authors consider entropy solutions which are entropic for the large family of Kruzhkov entropies $\{\eta_k\}_{k\in\mathbb{R}}$, where
\begin{align}
\eta_k(u)\coloneqq\abs{u-k}. \label{kruzkov_entropies}
\end{align}
See \cite{MR0267257}.
The paper \cite{MR3248030} gives a partial uniqueness: uniqueness for entropic  solutions to \eqref{burgert_hilbert_system} which are spatially periodic and have  locally bounded variation. On the contrary, the results we present in this paper work without smallness assumptions and do not require that solutions have locally bounded variation. This is due to the method of relative entropy not being a perturbative theory  -- smallness does not play a role in the relative entropy method or the theory of shifts and a-contraction.

Further, in \cite{MR3605552} the authors prove short-time existence and uniqueness for piecewise-smooth solutions $\bar{u}$ to \eqref{burgert_hilbert_system} which have a single discontinuity in space, along a shock curve in space-time. They consider `entropic' solutions of the form
\begin{align}
\begin{cases}\label{form_for_burgers_hilbert_paper}
\hspace{1.5in}\bar{u}(x,t)=\phi(x-s(t))+w(x-s(t),t),\\
\mbox{where $t\mapsto s(t)$ is the location of the shock, $w(\cdot,t)\in H^2(\mathbb{R}\setminus\{0\})$, and}\\
\hspace{1.5in}\phi(x)\coloneqq \frac{2}{\pi}\abs{x}\ln(\abs{x})m(x)\\
\mbox{for the smooth bump function with compact support $m\in C^{\infty}_{\mathcal{C}}(\mathbb{R})$, with $m(x)=m(-x)$, }\\ \mbox{and verifying}\\
\hspace{1.5in}\begin{cases}
m(x)=1 &\hspace{.3in}\mbox{if}\hspace{.2in} \abs{x}\leq 1,\\
m(x)=0 &\hspace{.3in}\mbox{if}\hspace{.2in} \abs{x}\geq 2,\\
m'(x)\leq0 &\hspace{.3in}\mbox{if}\hspace{.2in} x\in[1,2].
\end{cases}\\
\mbox{Note, $\phi$ has compact support in $[-2,2]$ and is smooth for $x\neq0$. Further, their definition}\\ \mbox{of a solution $\bar{u}$ requires that $\bar{u}(s(t)-,t)>\bar{u}(s(t)+,t)$ and $\norm{w(\cdot,t)}_{H^2(\mathbb{R}\setminus\{0\})}$ is bounded}\\ \mbox{uniformly in time (see \cite[p.~168]{MR3605552} for details). Their solutions are smooth at}\\ \mbox{$(x,t)$ whenever $x\neq s(t)$.}
\end{cases}
\end{align}
The form \eqref{form_for_burgers_hilbert_paper} is the class of solutions in \cite{MR3605552} on which uniqueness for piecewise-smooth solutions is shown. 

Note that smooth solutions to \eqref{system} will be entropic (see \eqref{entropy_condition_distributional_system_chitchat}) for any strictly convex  entropy $\eta$, and will in fact satisfy \eqref{entropy_condition_distributional_system_chitchat} as an exact equality. Moreover, solutions of the form \eqref{form_for_burgers_hilbert_paper} will also be entropic for any strictly convex entropy (see \Cref{burgers_hilbert_solutions_are_entropic}).

The difficulty in the analysis of the Burgers--Hilbert equation is that for the classical conservation laws without a source term, there is the well-known $L^1$-contractive semigroup property for solutions. Adding a source term which is Lipschitz continuous from $L^1\to L^1$ allows the semigroup to stay continuous (see \cite{MR3248030}). However, the Hilbert transform is a bounded linear operator on $L^2$ and not on $L^1$.

We apply the relative entropy method. The relative entropy method is well adapted to handling this kind of source term, since the relative entropy method is an $L^2$-based theory.

In this paper, we consider solutions which are entropic for one strictly convex entropy $\eta$. In the case of $G=\mathbf{H}$ (the Hilbert transform), we consider the unique entropy $\eta(u)=\frac{u^2}{2}$. If $G$ is a possibly nonlinear bounded operator from $L^\infty$ to $L^\infty$ we can consider an $\eta$ which is any strictly convex entropy. To be precise, we consider $\eta$ and $G$ such that at least one of the following hold:
\begin{align}\label{requirements_G_piecewise_scalar}
\begin{cases}
\bullet\mbox{ }\eta(u)= \alpha u^2+\beta u+\gamma\mbox{ for some }\alpha\in(0,\infty), \beta\in\mathbb{R},\gamma\in\mathbb{R},\\
\mbox{or}\\
\bullet\mbox{ }$G$\mbox{ (from the right hand side of \eqref{system}) is bounded from}\\
L^\infty(\mathbb{R})\to L^\infty(\mathbb{R}).
\end{cases}
\end{align}

The requirement \eqref{requirements_G_piecewise_scalar} is due to technical concerns.

In this paper, we prove the following theorem for scalar balance laws in the form \eqref{system}, for stability of solutions $\bar{u}$ in the form \eqref{form_for_burgers_hilbert_paper} which are smooth on $\{(x,t)\in\mathbb{R}\times[0,T) | x<s(t)\}$ and on $\{(x,t)\in\mathbb{R}\times[0,T) | x>s(t)\}$, where $s:[0,T)\to\mathbb{R}$ is a Lipschitz function.

The result we prove is:
\begin{theorem}[Main theorem -- $L^2$ stability for entropic piecewise-Lipschitz solutions to scalar balance laws]\label{main_stability_result_statement_scalar}

 Fix $T>0$.

Consider $u,\bar{u}$ weak solutions to \eqref{system}. Assume  $u \in L^2(\mathbb{R}\times[0,T))\cap L^\infty(\mathbb{R}\times[0,T))$ verifies the strong trace property (\Cref{strong_trace_definition}) and is entropic for the strictly convex entropy $\eta\in C^3(\mathbb{R})$, where $\eta$ and $G$ verify \eqref{requirements_G_piecewise_scalar}. Further, assume $\bar{u}$ is in the form \eqref{form_for_burgers_hilbert_paper}.

Assume also that there exists $\delta>0$ such that 
\begin{align}\label{delta_room_scalar}
\bar{u}(s(t)-,t)-\bar{u}(s(t)+,t)>\delta
\end{align}
for all $t\in[0,T)$.

Then, there exists constants $C>0$ and $\rho,\gamma>1$ and a Lipschitz continuous function $X:[0,T)\to\mathbb{R}$ with $X(0)=0$ such that for $b\in[0,T)$,
\begin{equation}
\begin{aligned}\label{main_global_stability_result_scalar_theorem_version}
&\hspace{-.2in}\int\limits_{\mathbb{R}}\eta(u(x,b)|\bar{u}(x+X(b),b))\,dx\\
&\leq
C \Bigg[\Bigg(\int\limits_{\mathbb{R}}\eta(u(x,0)|\bar{u}(x,0))\,dx\Bigg)^{1/\gamma}
\\
&\hspace{1in}+
\Bigg(\int\limits_{\mathbb{R}}\eta(u(x,0)|\bar{u}(x,0))\,dx\Bigg)^\rho\Bigg] e^{Cb+
C\int\limits_0^{b}\norm{\partial_ x\bar{u}(\cdot,t)}_{L^2(\mathbb{R}\setminus\{x=s(t)\})}^2\,dt}.
\end{aligned}
\end{equation}

Moreover, 
\begin{align}
X(t)=s(t)-h(t),
\end{align}
where $h(t)$ is a generalized characteristic of $u$, and we have the following $L^2$-type control on $X$:
\begin{align}\label{L2_control_on_shift_theorem}
&\int\limits_{0}^{T}(\dot{X}(t))^2\,dt\leq \tilde{C}\Bigg[\Bigg(\int\limits_{\mathbb{R}}\eta(u(x,0)|\bar{u}(x,0))\,dx\Bigg)^{1/\tilde{\gamma}}
\\
&\hspace{2in}+
\Bigg(\int\limits_{\mathbb{R}}\eta(u(x,0)|\bar{u}(x,0))\,dx\Bigg)^{\tilde{\rho}}\Bigg],
\end{align}
for constants $\tilde{C}>0$ and $\tilde{\rho},\tilde{\gamma}>1$.

The constants $C,\gamma,\rho$ depend on $T$ and $\delta$. The constants $\tilde{C},\tilde{\gamma},\tilde{\rho}$ depend on $T$, $\delta$ and $\int\limits_0^{T}\norm{\partial_ x\bar{u}(\cdot,t)}_{L^2(\mathbb{R}\setminus\{x=s(t)\})}^2\,dt$.

\end{theorem}
\begin{remark}
\hfill
\begin{itemize}
\item
The map $[0,T)\ni t\mapsto \norm{\partial_ x\bar{u}(\cdot,t)}_{L^2(\mathbb{R}\setminus\{x=s(t)\})}^2$ is in $L^1([0,T))$ because $\partial_x \bar{u} \in  L^2((\mathbb{R}\times[0,T))\setminus\{(x,t)|x=s(t)\})$ by the assumption that $\bar{u}$ is in the form \eqref{form_for_burgers_hilbert_paper}.
\item
Although the main application for \Cref{main_stability_result_statement_scalar} is to take $G=\mathbf{H}$ the Hilbert transform, \Cref{main_stability_result_statement_scalar} holds for all translation invariant, Lipschitz continuous $G$ and whenever $\bar{u}$ is in the form \eqref{form_for_burgers_hilbert_paper}.

\item
Instead of \eqref{requirements_G_piecewise_scalar}, the proof of\Cref{main_stability_result_statement_scalar} will actually go through whenever we have an estimate of the form
\begin{align}\label{generalized_requirement_remark_scalar_paper}
\abs{\int\limits_{x_1}^{x_2} \eta'(u(x,t)|\bar{u}(x+X(t),t))G(u(\cdot,t))(x)\,dx} \leq C \int\limits_{x_1}^{x_2} \abs{\eta'(u(x,t)|\bar{u}(x+X(t),t))}\,dx,
\end{align}
for $x_1,x_2\in\mathbb{R}$ and some constant $C>0$. Remark that for $u\in L^\infty$, \eqref{requirements_G_piecewise_scalar} implies \eqref{generalized_requirement_remark_scalar_paper}. This is due to $\eta'(a|b)\equiv 0$ for all $a,b\in\mathbb{R}$ whenever $\eta$ is a quadratic polynomial.
\item
If $\bar{u}$ is Lipschitz continuous on $\mathbb{R}\setminus\{x=s(t)\}$ and there exists a constant $C>0$ such that 
\begin{align}
\norm{\partial_ x\bar{u}(\cdot,t)}_{L^2(\mathbb{R}\setminus\{x=s(t)\})}\leq C
\end{align}
 for all $t\in[0,T)$ , which in particular are common assumption for solutions to the scalar conservation laws without a source term, then with slight modifications to the argument used to prove \Cref{main_stability_result_statement_scalar} (in fact simpler), we can get stability estimates of the form 

\begin{align}\label{main_global_stability_result_scalar_if_lambda4}
\int\limits_{\mathbb{R}}\abs{u(x,t_0)-\bar{u}(x+X(t_0),t_0)}^2\,dx\leq Ce^{Ct_0}\int\limits_{\mathbb{R}}\abs{u^0(x)-\bar{u}^0(x)}^2\,dx,
\end{align}
for a constant $C>0$ and for all $t_0>0$. Further, we get control on the shift in the form of $L^2$-type control on $X$:
\begin{align}\label{control_on_shift_piecewise_scalar_L2}
\int\limits_0^{t_0}(\dot{X}(t))^2\,dt \leq C\int\limits_{\mathbb{R}}\abs{u^0(x)-\bar{u}^0(x)}^2\,dx,
\end{align}
for all $t_0>0$.
These calculations are done explicitly for the systems case in a forthcoming paper \cite{move_entire_solution_system}. When $\bar{u}$ is Lipschitz continuous on $\big(\mathbb{R}\setminus\{x=s(t)\}\big)\times[0,\infty)$, the constant $C$ does not depend on the time interval $[0,T)$ on which the solution exists. Note that \eqref{control_on_shift_piecewise_scalar_L2} shows that the shift $X$ converges to zero  as $t\to\infty$ whenever $u^0-\bar{u}^0\in L^2(\mathbb{R})$. Further, note that the estimate \eqref{control_on_shift_piecewise_scalar_L2} can be applied to the scalar conservation law \eqref{system} (with $G\equiv0$) and when $\bar{u}$ is a pure shock wave with two constant states, \eqref{control_on_shift_piecewise_scalar_L2} shows that all of the generalized characteristics of $u$ eventually converge in speed to the shock speed of $\bar{u}$ if $u^0-\bar{u}^0\in L^2(\mathbb{R})$ (by noting that translations of $u$ are also $L^2$-close to $\bar{u}$ we can study all the characteristics of $u$, not just the characteristics $h$ verifying $h(0)=s(0)$). In particular, we have given an answer to Leger's conjecture on the large-time behavior of the shift \cite[p.~763]{Leger2011_original}. 
\item From H\"older duality, note that \eqref{control_on_shift_piecewise_scalar_L2} implies an estimate of the form
 \begin{align}\label{control_on_shift_piecewise_scalar_L2_average}
\frac{1}{t_0}\int\limits_0^{t_0}\abs{\dot{X}(t)}\,dt \leq \frac{C}{\sqrt{t_0}}\norm{u^0(\cdot)-\bar{u}^0(\cdot)}_{L^2(\mathbb{R})},
\end{align} 
and similarly \eqref{L2_control_on_shift_theorem} implies  
 \begin{align}\label{control_on_shift_piecewise_scalar_L2_average1}
&\frac{1}{T}\int\limits_0^{T}\abs{\dot{X}(t)}\,dt \leq \frac{C}{\sqrt{T}}\Bigg[\Bigg(\int\limits_{\mathbb{R}}\eta(u(x,0)|\bar{u}(x,0))\,dx\Bigg)^{1/\tilde{\gamma}}
\\
&\hspace{2in}+
\Bigg(\int\limits_{\mathbb{R}}\eta(u(x,0)|\bar{u}(x,0))\,dx\Bigg)^{\tilde{\rho}}\Bigg]^{1/2}.
\end{align} 
\end{itemize}
\end{remark}

By applying \Cref{main_stability_result_statement_scalar} with $A(u)=\frac{u^2}{2}$, $\eta(u)=\frac{u^2}{2}$, and $G$ the Hilbert transform, we are able to show that the piecewise-smooth solutions to the Burgers--Hilbert equation are unique among a much larger class of solutions than those considered in \cite{MR3605552}: For a fixed $T>0$, we show uniqueness for the solutions from\cite{MR3605552} amongst all weak solutions $u$ to \eqref{burgert_hilbert_system} verifying the strong trace property (\Cref{strong_trace_definition}), $u\in L^2(\mathbb{R}\times[0,T))\cap L^\infty(\mathbb{R}\times[0,T))$, initial data $u^0\in L^2(\mathbb{R})$, and entropic for $\eta(u)=\frac{u^2}{2}$. 

Note we can apply \Cref{main_stability_result_statement_scalar} to the piecewise-smooth solutions constructed in \cite{MR3605552} (the $\bar{u}$ in the context of \Cref{main_stability_result_statement_scalar}) because  due to the smoothness of the solutions in \cite{MR3605552}, for each finite time $T$ we can use compactness to find a $\delta$ to verify \eqref{delta_room_scalar}.

In particular, we do not require the lines worth of entropies \eqref{kruzkov_entropies} to show uniqueness. By virtue of the relative entropy method, we use only a single entropy condition. The single entropy condition for the classic Burgers equation is itself a relatively new development (see \cite{panov_uniquness,delellis_uniquneness,2017arXiv170905610K}) and is of interest because for systems of conservation laws or balance laws, often only one nontrivial entropy exists, and we certainly cannot hope for all of the  Kruzhkov entropies \eqref{kruzkov_entropies}. In particular, due to our use of the relative entropy method, there is hope that the techniques developed in this paper will extend to systems.

Our stability and uniqueness result \Cref{main_stability_result_statement_scalar} is sharp in the following sense: in \cite{MR3248030} they construct a source term $G$ for \eqref{system}  which is Lipschitz continuous $L^2\to L^2$. They use the quadratic flux, $A(u)=\frac{u^2}{2}$. However, for this $G$ they construct initial data which has multiple solutions, all with bounded variation and uniformly compact support. They conjecture then that any proof of uniqueness for \eqref{burgert_hilbert_system} must rely on specific properties of the Hilbert transform. And indeed, our proof of \Cref{main_stability_result_statement_scalar} relies on the translation invariance of $G$, a property also of the Hilbert transform.

As noted in \cite{MR3605552} and \cite{MR3248030}, the well-posedness and uniqueness of solutions for \eqref{burgert_hilbert_system} remain largely open questions.

All the previous results in the program of stability up to a shift work by rewriting the piecewise smooth $\bar{u}$ as two solutions to \eqref{system}, both defined on all of $\mathbb{R}$, and then letting the shift function decide how much of each solution (in $x$) to use at each particular time $t$. This is in particular how \cite{2017arXiv170905610K} works. Creating these two different solutions is often easy enough when the $\bar{u}$ is piecewise constant or \eqref{system} is scalar. But it is less obvious how to create these two separate solutions when \eqref{system} has multiple conserved quantities or when \eqref{system} is nonlocal. When \eqref{system} is nonlocal, we cannot simply extend each half (to the left/right of the jump discontinuity) of $\bar{u}$ to all of $\mathbb{R}$.

Thus, in this paper we introduce the idea of not moving the discontinuity in $\bar{u}$ to keep stability in $L^2$, but instead translating the entire $\bar{u}$ in $x$, as a function of $u$ and $t$. There is a key difficulty in this: translating $\bar{u}$ causes entropy production so large (see \Cref{global_entropy_dissipation_rate_systems}) we can no longer close the key Gronwall argument to get $L^2$ stability. To solve this problem, we give a novel construction of the shift function which actually creates negative entropy -- \emph{comparable in amount to the excess growth in the Gronwall} (see \Cref{negative_entropy_diss_scalar_lemma} and the proof of \Cref{main_stability_result_statement_scalar}). This closes the Gronwall argument and gives the main theorem, \Cref{main_stability_result_statement_scalar}.

The key insight into the construction of the new shift function is that for the scalar case, if $u(\cdot,t)$ (the function in the first slot of the relative entropy) and $\bar{u}(\cdot,t)$ (the function in the second slot) both have a discontinuity at $x_0\in\mathbb{R}$, then we can always kill the entropy growth between $\bar{u}$ and $u$ in $L^2$ by moving the discontinuity in $\bar{u}$ at the speed of the discontinuity in $u$. This is exactly what \Cref{negative_entropy_diss_scalar_lemma} is saying. In other words, the shift function is the generalized characteristic of $u$ (see Dafermos \cite[Chapter 10]{dafermos_big_book} for generalized characteristics). Further, this very natural shift function does more than neutralize all growth in $L^2$ that would otherwise occur due to the discontinuity in $\bar{u}$: it creates additional negative entropy. This negativity allows us to translate $\bar{u}$ in $x$. Further, this negative entropy is responsible for the novel $L^2$ control presented on the shifts in this paper (see \eqref{L2_control_on_shift_theorem} and \eqref{control_on_shift_piecewise_scalar_L2}). It is striking that we can use a generalized characteristic as a shift function.

Assume $\bar{u}$ has a discontinuity along the curve $x=s(t)$. Then when  computing $\partial_t\int\eta(u|\bar{u})\,dx$ using  \eqref{entropy_condition_distributional_system_chitchat}, the discontinuity in $\bar{u}$ dissipates entropy relative to $u$, in the amount of 
\begin{equation}
\begin{aligned}\label{example_dissipation}
&q(u_+;\bar{u}_+)-q(u_-;\bar{u}_-)-\sigma(\bar{u}_+,\bar{u}_-)\big(\eta(u_+|\bar{u}_+)-\eta(u_-|\bar{u}_-)\big),
\end{aligned}
\end{equation}
where $u_{\pm}\coloneqq u(s(t)\pm,t)$, $\bar{u}_{\pm}\coloneqq \bar{u}(s(t)\pm,t)$, and  $\sigma(\bar{u}_+,\bar{u}_-)$ is the speed of the shock connecting $\bar{u}_-$ and $\bar{u}_+$ (see \Cref{global_entropy_dissipation_rate_systems} for details).

Loosely speaking, previous shift functions, including as used in previous a-contraction papers and in \cite{Leger2011_original}, were more ad hoc and worked by making $\sigma(\bar{u}_+,\bar{u}_-)$ an unknown, setting the dissipation \eqref{example_dissipation} equal to zero and solving for $\sigma(\bar{u}_+,\bar{u}_-)$, giving something like
\begin{equation}
\begin{aligned}\label{idea_for_shift_feb2019_1}
&\hspace{.32in}\sigma(\bar{u}_+,\bar{u}_-)\hspace{.29in}\wanteq\hspace{.29in} \frac{\max\{0,q(u_+;\bar{u}_+)-q(u_-;\bar{u}_-)\}}{\eta(u_+|\bar{u}_+)-\eta(u_-|\bar{u}_-)}
\end{aligned}
\end{equation}
(see \cite{serre_vasseur} or \cite{Leger2011_original} for more details).
Then we study the map,
\begin{align}\label{set_to_0_map}
(u,u_L,u_R)\mapsto  \frac{\max\{0,q(u;u_R)-q(u;u_L)\}}{\eta(u|u_R)-\eta(u|u_L)}
\end{align}
(motivated by the case where $u(\cdot,t)$ is continuous at $s(t)$). Unfortunately, the map \eqref{set_to_0_map} is unyielding to analysis. Simple  bounds on this map are difficult to obtain (see for example the proof of Lemma 3.4 in \cite{2017arXiv170905610K}). Moreover, the map \eqref{set_to_0_map} is not continuous or  upper semi-continuous in $u$, which creates the need for a secondary mollification argument not needed in the present paper.

Further, by design, when $u(\cdot,t)$ is continuous at $s(t)$, the shift function created based on the above ideas (\eqref{idea_for_shift_feb2019_1} and \eqref{set_to_0_map}) will make the entropy growth between $u$ and $\bar{u}$ identically zero, which is not good for us. We want some additional negativity.

An intuitive way to see that the generalized characteristic of $u$ is the correct shift function is this:
Fix a $v\in\mathbb{R}$. Note that if $u$ is entropic for the entropy $\eta$, $u$ is also entropic for the entropy
\begin{align}\label{relative_entropy_with_v_scalar}
u\mapsto \eta(u|v).
\end{align}
This is due to the map \eqref{relative_entropy_with_v_scalar} being just  $\eta(u)$ plus a constant term and a term linear in $u$.

Then, if the shock $(u_L,u_R,\sigma(u_L,u_R))$ is entropic for $\eta$, we also have for \eqref{relative_entropy_with_v_scalar}:
\begin{align}\label{shock_entropic_rel_entropic_scalar_for_v}
q(u_R;v)-q(u_L;v)-\sigma(u_L,u_R)(\eta(u_R|v)-\eta(u_L|v))\leq0.
\end{align}
Note the speed in \eqref{shock_entropic_rel_entropic_scalar_for_v} is $\sigma(u_L,u_R)$ -- the speed of the discontinuity in the first slot of the relative entropy. The speed is \emph{not} $\sigma(v,v)=A'(v)$. 

\vspace{.07in}

The outline of the paper is as follows. First, in \Cref{tech_lemmas_sec} we present some technical lemmas and some structural lemmas on the system \eqref{system} which we will need. Then, in \Cref{neg_entrop_sec} we present the proof of the additional negative entropy the generalized characteristic causes when used as a shift function (see \Cref{negative_entropy_diss_scalar_lemma}). Finally, in \Cref{main_proof_sec} we translate the piecewise-smooth solution in $x$ artificially to prove \Cref{main_stability_result_statement_scalar}, and we cancel the entropy caused by the translation using the negative entropy from the generalized-characteristic-based shift function.

\section{Technical Lemmas}\label{tech_lemmas_sec}

Throughout this paper, we will use as notation and define the relative flux,
\begin{align}
\label{Z_def}
A(u|\bar{u})\coloneqq A(u)-A(\bar{u})- A'(\bar{u})(u-\bar{u}).
\end{align}

We will also use the relative entropy \emph{derivative},
\begin{align}
\label{entropy_rel_deriv}
\eta'(u|\bar{u})\coloneqq \eta'(u)-\eta'(\bar{u})- \eta''(\bar{u})(u-\bar{u}).
\end{align}

For completeness, we state the following basic fact:
\begin{lemma}\label{burgers_hilbert_solutions_are_entropic}
A solution $\bar{u}$ to \eqref{system}  of the form \eqref{form_for_burgers_hilbert_paper} will be entropic, i.e. satisfy \eqref{entropy_condition_distributional_system_chitchat}, for any strictly convex entropy $\eta\in C^2(\mathbb{R})$.
\end{lemma}
\begin{proof}
This follows because smooth solutions to \eqref{system} will satisfy \eqref{entropy_condition_distributional_system_chitchat} as an exact equality for any smooth  entropy $\eta$. Further, solutions of the form \eqref{form_for_burgers_hilbert_paper} are smooth except along the shock curve $t\mapsto s(t)$, along which they verify $\bar{u}(s(t)-,t)>\bar{u}(s(t)+,t)$. Thus, \Cref{entropic_shock_lemma_scalar_lax} ensures that $\bar{u}$ satisfies \eqref{entropy_condition_distributional_system_chitchat}.
\end{proof}

The following lemma gives the relative entropy dissipation produced by three sources: a shock, the source term $G$, and  translating the piecewise smooth solution $\bar{u}$ by a function $t\mapsto X(t)$.
\begin{lemma}[Global entropy dissipation rate]\label{global_entropy_dissipation_rate_systems}

Fix $T>0$.

Consider $u,\bar{u}$ weak solutions to \eqref{system}. Assume  $u \in L^2(\mathbb{R}\times[0,T))\cap L^\infty(\mathbb{R}\times[0,T))$ verifies the strong trace property (\Cref{strong_trace_definition}) and is entropic for the strictly convex entropy $\eta\in C^3(\mathbb{R})$. Further, assume $\bar{u}$ is in the form \eqref{form_for_burgers_hilbert_paper}. Also, assume that $u^0-\bar{u}^0\in L^2(\mathbb{R})$.

Let  $h:[0,T)\to\mathbb{R}$ be a Lipschitz continuous map.

Define
\begin{align}
X(t)\coloneqq s(t)-h(t),\label{X_def_calc}
\end{align}
where $s$ is as in \eqref{form_for_burgers_hilbert_paper}.

Then for almost every $a,b\in[0,T)$ verifying $a<b$, 
\begin{equation}
\begin{aligned}\label{nonlocal_dissipation_formula}
&\int\limits_{-\infty}^{\infty}\eta(u(x,b)|\bar{u}(x+X(b),b))\,dx-\int\limits_{-\infty}^{\infty}\eta(u(x,a)|\bar{u}(x+X(a),a))\,dx
\\
&\hspace{.7in}\leq\int\limits_{a}^{b}q(u(h(t)+,t);\bar{u}(s(t)+,t))-q(u(h(t)-,t);\bar{u}(s(t)-,t))
\\
&\hspace{.7in}-\dot{h}(t)\big(\eta(u(h(t)+,t)|\bar{u}(s(t)+,t))-\eta(u(h(t)-,t)|\bar{u}(s(t)-,t))\big)\,dt
\\
&\hspace{.7in}-\int\limits_{a}^{b} \int\limits_{-\infty}^{\infty}\Bigg[\Bigg(\partial_x \bigg|_{(x+X(t),t)}\hspace{-.45in} \eta'(\bar{u}(x,t))\Bigg) A(u(x,t)|\bar{u}(x+X(t),t))
\\
&\hspace{.7in}+\Bigg(2\partial_x\bigg|_{(x+X(t),t)}\hspace{-.45in}\bar{u}(x,t)\dot{X}(t)\Bigg)\eta''(\bar{u}(x+X(t),t))[u(x,t)-\bar{u}(x+X(t),t)]
\\
&\hspace{.7in}-\eta'(u(x,t)|\bar{u}(x+X(t),t))G(u(\cdot,t))(x)
\\
&\hspace{-.2in}+
\Bigg(G(\bar{u}(\cdot,t))(x+X(t))-G(u(\cdot,t))(x)\Bigg)\eta''(\bar{u}(x+X(t),t))[u(x,t)-\bar{u}(x+X(t),t)]\Bigg]\,dxdt.
\end{aligned}
\end{equation}

Moreover, \eqref{nonlocal_dissipation_formula} holds for $a=0$  and almost every $b\in(0,T)$.

\end{lemma}

\begin{proof}
\hfill\break
\uline{Step 1}
\hfill\break

We first prove the following relative version of the entropy inequality:

For all positive, Lipschitz continuous test functions $\phi:\mathbb{R}\times[0,T)\to\mathbb{R}$ with compact support and that vanish on the set $\{(x,t)\in\mathbb{R}\times[0,T) | x=s(t)-X(t)\}$, we have
\begin{equation}
\begin{aligned}\label{combined1}
&\int\limits_{0}^{T} \int\limits_{-\infty}^{\infty} [\partial_t \phi \eta(u(x,t)|\bar{u}(x+X(t),t))+\partial_x \phi q(u(x,t);\bar{u}(x+X(t),t))]\,dxdt 
\\&\hspace{.2in}+  \int\limits_{-\infty}^{\infty}\phi(x,0)\eta(u^0(x)|\bar{u}^0(x))\,dx 
\geq
\\
&\hspace{1.1in}\int\limits_{0}^{T} \int\limits_{-\infty}^{\infty}\phi\Bigg[\Bigg(
\partial_x \bigg|_{(x+X(t),t)} \hspace{-.45in}\eta'(\bar{u}(x,t))\Bigg)A(u(x,t)|\bar{u}(x+X(t),t))
\\
&\hspace{1.1in}+\Bigg(2\partial_x\bigg|_{(x+X(t),t)}\hspace{-.45in}\bar{u}(x,t)\dot{X}(t)\Bigg)\eta''(\bar{u}(x+X(t),t))[u(x,t)-\bar{u}(x+X(t),t)]
\\
&\hspace{1.1in}-\eta'(u(x,t)|\bar{u}(x+X(t),t))G(u(\cdot,t))(x)
\\
&+
\Bigg(G(\bar{u}(\cdot,t))(x+X(t))-G(u(\cdot,t))(x)\Bigg)\eta''(\bar{u}(x+X(t),t))[u(x,t)-\bar{u}(x+X(t),t)]
\Bigg]\,dxdt.
\end{aligned}
\end{equation}

Note that \eqref{combined1} is the analogue in our case of the key estimate used in Dafermos's proof of weak-strong stability, which gives a relative version of the entropy inequality (see equation (5.2.10) in \cite[p.~122-5]{dafermos_big_book}). The proof of \eqref{combined1} is based on the famous weak-strong stability proof of Dafermos and DiPerna \cite[p.~122-5]{dafermos_big_book}.

We now prove \eqref{combined1}.

%

Note that on the complement of the set $\{(x,t)\in\mathbb{R}\times[0,T) | x=s(t)\}$, $\bar{u}$ is smooth  and so we have the exact equalities,
\begin{align}
\partial_t\bigg|_{(x,t)}\hspace{-.21in}\big(\bar{u}(x,t)\big)+\partial_x\bigg|_{(x,t)}\hspace{-.21in}\big(A(\bar{u}(x,t))\big)&=G(\bar{u}(\cdot,t))(x),\label{solves_equation}\\
\partial_t\bigg|_{(x,t)}\hspace{-.21in}\big(\eta(\bar{u}(x,t))\big)+\partial_x\bigg|_{(x,t)}\hspace{-.21in}\big(q(\bar{u}(x,t))\big)&=\eta'(\bar{u}(x,t))G(\bar{u}(\cdot,t))(x).\label{solves_entropy}
\end{align}

Thus for any Lipschitz continuous function $X: [0,T)\to\mathbb{R}$ with $X(0)=0$  we have on the complement of the set $\{(x,t)\in\mathbb{R}\times[0,T) | x=s(t)-X(t)\}$, 
\begin{equation}
\begin{aligned}\label{solves_equation_shift}
\partial_t\bigg|_{(x,t)}\hspace{-.21in}&\big(\bar{u}(x+X(t),t)\big)+\partial_x\bigg|_{(x,t)}\hspace{-.21in}\big(A(\bar{u}(x+X(t),t))\big)=
\\
&\hspace{1.5in}\Bigg(\partial_x\bigg|_{(x+X(t),t)}\hspace{-.45in}\big(\bar{u}(x,t)\big)\Bigg)\dot{X}(t)+G(\bar{u}(\cdot,t))(x+X(t)),
\end{aligned}
\end{equation}
and
\begin{equation}
\begin{aligned}\label{solves_entropy_shift}
\partial_t\bigg|_{(x,t)}\hspace{-.21in}&\big(\eta(\bar{u}(x+X(t),t))\big)+\partial_x\bigg|_{(x,t)}\hspace{-.21in}\big(q(\bar{u}(x+X(t),t))\big)=
\\
&\eta'(\bar{u}(x+X(t),t))\Bigg(\partial_x\bigg|_{(x+X(t),t)}\hspace{-.45in}\big(\bar{u}(x,t)\big)\Bigg)\dot{X}(t)+\eta'(\bar{u}(x+X(t),t))G(\bar{u}(\cdot,t))(x+X(t)).
\end{aligned}
\end{equation}

We can now imitate the weak-strong stability proof in \cite[p.~122-5]{dafermos_big_book}, using \eqref{solves_equation_shift} and \eqref{solves_entropy_shift} instead of \eqref{solves_equation} and \eqref{solves_entropy}.

Recall \eqref{Z_def}, which says

\begin{align}
A(u|\bar{u})\coloneqq A(u)-A(\bar{u})-A' (\bar{u})(u-\bar{u}).
\end{align}
Remark that $A(u|\bar{u})$ is locally quadratic in $u-\bar{u}$. 

Fix any positive, Lipschitz continuous test function $\phi:\mathbb{R}\times[0,T)\to\mathbb{R}$ with compact support. Assume also that $\phi$ vanishes on the set $\{(x,t)\in\mathbb{R}\times[0,T) | x=s(t)-X(t)\}$. Then, we use that $u$ satisfies the entropy inequality in a distributional sense:
 \begin{equation}
\begin{aligned}\label{u_entropy_integral_formulation}
\int\limits_{0}^{T} \int\limits_{-\infty}^{\infty}\Bigg[\partial_t\phi\big(\eta(u(x,t))\big)+&\partial_x \phi \big(q(u(x,t))\big)\Bigg]\,dxdt+ \int\limits_{-\infty}^{\infty}\phi(x,0)\eta(u^0(x))\,dx\geq
\\
&-\int\limits_{0}^{T} \int\limits_{-\infty}^{\infty}\phi\eta'(u(x,t))G(u(\cdot,t))(x)\,dxdt
\end{aligned}
\end{equation}

 We also view \eqref{solves_entropy_shift} as a distributional equality:
 \begin{equation}
\begin{aligned}\label{solves_entropy_shift_integral_formulation}
\int\limits_{0}^{T} \int\limits_{-\infty}^{\infty}\Bigg[\partial_t\phi\big(\eta(\bar{u}(x+&X(t),t))\big)+\partial_x \phi \big(q(\bar{u}(x+X(t),t))\big)\Bigg]\,dxdt+ \int\limits_{-\infty}^{\infty}\phi(x,0)\eta(\bar{u}^0(x))\,dx=
\\
&-\int\limits_{0}^{T} \int\limits_{-\infty}^{\infty}\phi\Bigg[\eta'(\bar{u}(x+X(t),t))\Bigg(\partial_x\bigg|_{(x+X(t),t)}\hspace{-.45in}\big(\bar{u}(x,t)\big)\Bigg)\dot{X}(t)+
\\
&\hspace{1.3in}\eta'(\bar{u}(x+X(t),t))G(\bar{u}(\cdot,t))(x+X(t))\Bigg]\,dxdt.
\end{aligned}
\end{equation}

To get \eqref{solves_entropy_shift_integral_formulation}, we do integration by parts twice on the right hand side of \eqref{solves_entropy_shift}. Once on the domain $\{(x,t)\in\mathbb{R}\times[0,T) | x<s(t)-X(t)\}$ and once on the domain $\{(x,t)\in\mathbb{R}\times[0,T) | x>s(t)-X(t)\}$. We don't have a boundary term along the set $\{(x,t)\in\mathbb{R}\times[0,T) | x=s(t)-X(t)\}$ because $\phi$ vanishes on this set.

 We subtract \eqref{solves_entropy_shift_integral_formulation} from \eqref{u_entropy_integral_formulation}, to get

\begin{equation}
\begin{aligned}\label{difference_entropy_equations}
&\int\limits_{0}^{T} \int\limits_{-\infty}^{\infty} [\partial_t \phi \eta(u(x,t)|\bar{u}(x+X(t),t))+\partial_x \phi q(u(x,t);\bar{u}(x+X(t),t))]\,dxdt \\
&\hspace{2in}+  \int\limits_{-\infty}^{\infty}\phi(x,0)\eta(u^0(x)|\bar{u}^0(x))\,dx \\
&\hspace{1in}\geq -\int\limits_{0}^{T} \int\limits_{-\infty}^{\infty} \Big(\partial_t \phi\eta'(\bar{u}(x+X(t),t))[u(x,t)-\bar{u}(x+X(t),t)]+
\\
&\hspace{1.5in}\partial_x \phi \eta'(\bar{u}(x+X(t),t))[A(u(x,t))-A(\bar{u}(x+X(t),t))]\Big)\,dxdt 
\\
&\hspace{1.5in}- \int\limits_{-\infty}^{\infty}\phi(x,0)\eta'(\bar{u}^0(x))[u^0(x)-\bar{u}^0(x)]\,dx
\\
&\hspace{1.5in}+
\int\limits_{0}^{T} \int\limits_{-\infty}^{\infty}\phi\Bigg[\eta'(\bar{u}(x+X(t),t))\Bigg(\partial_x\bigg|_{(x+X(t),t)}\hspace{-.45in}\big(\bar{u}(x,t)\big)\Bigg)\dot{X}(t)+
\\
&\hspace{1in}\eta'(\bar{u}(x+X(t),t))G(\bar{u}(\cdot,t))(x+X(t))-\eta'(u(x,t))G(u(\cdot,t))(x)\Bigg]\,dxdt
\end{aligned}
\end{equation}

The function $u$ is a distributional solution to the system of conservation laws. Thus, for every Lipschitz continuous test function $\Phi:\mathbb{R}\times[0,T)\to \mathbb{R}$ with compact support,
\begin{equation}
\begin{aligned}\label{u_solves_equation_integral_formulation}
\int\limits_{0}^{T} \int\limits_{-\infty}^{\infty} \Bigg[\partial_t\Phi u + \partial_x\Phi A(u) \Bigg]\,dxdt +\int\limits_{-\infty}^{\infty} \Phi(x,0)u^0(x)\,dx
\\
=-\int\limits_{0}^{T} \int\limits_{-\infty}^{\infty}\Phi G(u(\cdot,t))(x)\,dxdt.
\end{aligned}
\end{equation}

We also can rewrite \eqref{solves_equation_shift} in a distributional way, for $\Phi$ which have the additional property of vanishing on $\{(x,t)\in\mathbb{R}\times[0,T) | x=s(t)-X(t)\}$:
\begin{equation}
\begin{aligned}\label{shift_solves_equation_integral_formulation}
\int\limits_{0}^{T} \int\limits_{-\infty}^{\infty} \Bigg[\partial_t\Phi \bar{u}(x+X(t),t) + \partial_x\Phi A(\bar{u}(x+X(t),t)) \Bigg]\,dxdt +\int\limits_{-\infty}^{\infty} \Phi(x,0)\bar{u}^0(x)\,dx
\\
=-\int\limits_{0}^{T} \int\limits_{-\infty}^{\infty}\Phi \Bigg[\Bigg(\partial_x\bigg|_{(x+X(t),t)}\hspace{-.45in}\big(\bar{u}(x,t)\big)\Bigg)\dot{X}(t)+G(\bar{u}(\cdot,t))(x+X(t))\Bigg]\,dxdt.
\end{aligned}
\end{equation}
To prove \eqref{shift_solves_equation_integral_formulation}, on the right hand side of \eqref{solves_equation_shift} we again do integration by parts twice. Once on the domain $\{(x,t)\in\mathbb{R}\times[0,T) | x<s(t)-X(t)\}$ and once on the domain $\{(x,t)\in\mathbb{R}\times[0,T) | x>s(t)-X(t)\}$. We lose the boundary terms along $\{(x,t)\in\mathbb{R}\times[0,T) | x=s(t)-X(t)\}$ because $\Phi$ vanishes there. 

Then, we can choose 
\begin{align}\label{our_choice_for_Phi}
\phi\eta'(\bar{u}(x+X(t),t))
\end{align} 
 as the test function $\Phi$, and subtract \eqref{shift_solves_equation_integral_formulation} from \eqref{u_solves_equation_integral_formulation}. We can extend the function \eqref{our_choice_for_Phi} to the set $\{(x,t)\in\mathbb{R}\times[0,T) | x=s(t)-X(t)\}$ by defining it to be zero. This extension is still Lipschitz continuous.

This yields,
\begin{equation}
\begin{aligned}\label{difference_of_solutions_equation}
&\int\limits_{0}^{T} \int\limits_{-\infty}^{\infty} \Bigg[\partial_t[\phi\eta'(\bar{u}(x+X(t),t))][u(x,t)- \bar{u}(x+X(t),t)] + 
\\
&\hspace{1in}\partial_x[\phi\eta'(\bar{u}(x+X(t),t))][A(u(x,t))-A(\bar{u}(x+X(t),t))] \Bigg]\,dxdt 
\\&\hspace{1in}+\int\limits_{-\infty}^{\infty} \phi(x,0)\eta'(\bar{u}^0(x))[u^0(x)-\bar{u}^0(x)]\,dx
\\
&\hspace{1.5in}=\int\limits_{0}^{T} \int\limits_{-\infty}^{\infty}\phi\eta'(\bar{u}(x+X(t),t)) \Bigg[\Bigg(\partial_x\bigg|_{(x+X(t),t)}\hspace{-.45in}\big(\bar{u}(x,t)\big)\Bigg)\dot{X}(t)
\\&\hspace{1.5in}+G(\bar{u}(\cdot,t))(x+X(t))-G(u(\cdot,t))(x)\Bigg]\,dxdt.
\end{aligned}
\end{equation}

Recall $\bar{u}$ is a classical solution on the complement of the set $\{(x,t)\in\mathbb{R}\times[0,T) | x=s(t)\}$ and verifies \eqref{solves_equation_shift}. Thus, on the complement of the set $\{(x,t)\in\mathbb{R}\times[0,T) | x=s(t)-X(t)\}$,
\begin{equation}
\begin{aligned}\label{notice_this}
&\partial_t\bigg|_{(x,t)}\hspace{-.21in}\big(\eta'(\bar{u}(x+X(t),t))\big)=\Bigg(\partial_x\bigg|_{(x+X(t),t)}\hspace{-.45in}\bar{u}(x,t)\dot{X}(t)+\partial_t\bigg|_{(x+X(t),t)}\hspace{-.45in}\bar{u}(x,t)\Bigg)\eta''(\bar{u}(x+X(t),t))
\\
&\hspace{1.5in}=\Bigg(2\partial_x\bigg|_{(x+X(t),t)}\hspace{-.45in}\bar{u}(x,t)\dot{X}(t)
\\
&\hspace{.7in}-\partial_x \bigg|_{(x+X(t),t)}\hspace{-.45in} \bar{u}(x,t)\big[A'(\bar{u}(x+X(t),t))\big]+G(\bar{u}(\cdot,t))(x+X(t))\Bigg)\eta''(\bar{u}(x+X(t),t))
\\
&\hspace{1.5in}=\Bigg(2\partial_x\bigg|_{(x+X(t),t)}\hspace{-.45in}\bar{u}(x,t)\dot{X}(t)
+G(\bar{u}(\cdot,t))(x+X(t))\Bigg)\eta''(\bar{u}(x+X(t),t))
\\
&\hspace{2in}-\partial_x \bigg|_{(x+X(t),t)}\hspace{-.45in} \bar{u}(x,t)\eta''(\bar{u}(x+X(t),t))A'(\bar{u}(x+X(t),t)).
\end{aligned}
\end{equation}

Thus, by \eqref{notice_this} and the definition of the relative flux in \eqref{Z_def},

\begin{equation}
\begin{aligned}\label{5.2.9}
&\partial_t\bigg|_{(x,t)}\hspace{-.21in}\big(\eta'(\bar{u}(x+X(t),t))\big)[u(x,t)-\bar{u}(x+X(t),t)]
\\
&\hspace{.2in}+\partial_x\bigg|_{(x,t)}\hspace{-.21in}\big(\eta'(\bar{u}(x+X(t),t))\big)[A(u(x,t))-A(\bar{u}(x+X(t),t))]
\\
&\hspace{.5in}=
\partial_x \bigg|_{(x+X(t),t)}\hspace{-.45in} \bar{u}(x,t)\eta''(\bar{u}(x+X(t),t)) A(u(x,t)|\bar{u}(x+X(t),t))
\\
&\hspace{.1in}+\Bigg(2\partial_x\bigg|_{(x+X(t),t)}\hspace{-.45in}\bar{u}(x,t)\dot{X}(t)
+G(\bar{u}(\cdot,t))(x+X(t))\Bigg)\eta''(\bar{u}(x+X(t),t))[u(x,t)-\bar{u}(x+X(t),t)].
\end{aligned}
\end{equation}

We combine \eqref{difference_entropy_equations}, \eqref{difference_of_solutions_equation}, and \eqref{5.2.9} to get
\begin{equation}
\begin{aligned}\label{combined}
&\int\limits_{0}^{T} \int\limits_{-\infty}^{\infty} [\partial_t \phi \eta(u(x,t)|\bar{u}(x+X(t),t))+\partial_x \phi q(u(x,t);\bar{u}(x+X(t),t))]\,dxdt \\
&\hspace{.5in}+  \int\limits_{-\infty}^{\infty}\phi(x,0)\eta(u^0(x)|\bar{u}^0(x))\,dx \\
&\hspace{.8in}\geq
\int\limits_{0}^{T} \int\limits_{-\infty}^{\infty}\phi\Bigg[\eta'(\bar{u}(x+X(t),t))\Bigg(\partial_x\bigg|_{(x+X(t),t)}\hspace{-.45in}\big(\bar{u}(x,t)\big)\Bigg)\dot{X}(t)+
\\
&\hspace{.8in}\eta'(\bar{u}(x+X(t),t))G(\bar{u}(\cdot,t))(x+X(t))-\eta'(u(x,t))G(u(\cdot,t))(x)
\\
&\hspace{.8in}+\Bigg(\partial_x \bigg|_{(x+X(t),t)}\hspace{-.45in} \bar{u}(x,t)\Bigg)\eta''(\bar{u}(x+X(t),t)) A(u(x,t)|\bar{u}(x+X(t),t))
\\
&\hspace{1.8in}+\Bigg(2\partial_x\bigg|_{(x+X(t),t)}\hspace{-.45in}\bar{u}(x,t)\dot{X}(t)
\\
&\hspace{.8in}+G(\bar{u}(\cdot,t))(x+X(t))\Bigg)\eta''(\bar{u}(x+X(t),t))[u(x,t)-\bar{u}(x+X(t),t)]
\\
&\hspace{.1in}-\eta'(\bar{u}(x+X(t),t)) \Big[\Bigg(\partial_x\bigg|_{(x+X(t),t)}\hspace{-.45in}\big(\bar{u}(x,t)\big)\Bigg)\dot{X}(t)+G(\bar{u}(\cdot,t))(x+X(t))-G(u(\cdot,t))(x)\Big]
\Bigg]\,dxdt
\\
&\hspace{.8in}=\int\limits_{0}^{T} \int\limits_{-\infty}^{\infty}\phi\Bigg[-\eta'(u(x,t))G(u(\cdot,t))(x)
\\
&\hspace{1.2in}+\Bigg(\partial_x \bigg|_{(x+X(t),t)}\hspace{-.45in} \bar{u}(x,t)\Bigg)\eta''(\bar{u}(x+X(t),t)) A(u(x,t)|\bar{u}(x+X(t),t))
\\
&\hspace{1.2in}+\Bigg(2\partial_x\bigg|_{(x+X(t),t)}\hspace{-.45in}\bar{u}(x,t)\dot{X}(t)
\\
&\hspace{1.3in}+G(\bar{u}(\cdot,t))(x+X(t))\Bigg)\eta''(\bar{u}(x+X(t),t))[u(x,t)-\bar{u}(x+X(t),t)]
\\
&\hspace{1.2in}-\eta'(\bar{u}(x+X(t),t)) \Big[-G(u(\cdot,t))(x)\Big]
\Bigg]\,dxdt.
\end{aligned}
\end{equation}

Note that we can add zero, to get
\begin{equation}
\begin{aligned}\label{note_that}
&-\eta'(u(x,t))G(u(\cdot,t))(x)+G(\bar{u}(\cdot,t))(x+X(t))\eta''(\bar{u}(x+X(t),t))[u(x,t)-\bar{u}(x+X(t),t)]
\\
&\hspace{.3in}-\eta'(\bar{u}(x+X(t),t)) \Big[-G(u(\cdot,t))(x)\Big]
\\
&=-G(u(\cdot,t))(x)\Bigg(\big(\eta'(u(x,t))\big)-\big(\eta'(\bar{u}(x+X(t),t))\big)-\eta''(\bar{u}(x+X(t),t))[u(x,t)-\bar{u}(x+X(t),t)]\Bigg)
\\
&\hspace{.3in}+
\Bigg(G(\bar{u}(\cdot,t))(x+X(t))-G(u(\cdot,t))(x)\Bigg)\eta''(\bar{u}(x+X(t),t))[u(x,t)-\bar{u}(x+X(t),t)]
\\
&=-G(u(\cdot,t))(x)(\eta'(u(x,t)|\bar{u}(x+X(t),t)))
\\
&\hspace{.3in}+
\Bigg(G(\bar{u}(\cdot,t))(x+X(t))-G(u(\cdot,t))(x)\Bigg)\eta''(\bar{u}(x+X(t),t))[u(x,t)-\bar{u}(x+X(t),t)]
\\
&=-\eta'(u(x,t)|\bar{u}(x+X(t),t))G(u(\cdot,t))(x)
\\
&\hspace{.3in}+
\Bigg(G(\bar{u}(\cdot,t))(x+X(t))-G(u(\cdot,t))(x)\Bigg)\eta''(\bar{u}(x+X(t),t))[u(x,t)-\bar{u}(x+X(t),t)].
\end{aligned}
\end{equation}
Note that this computation is from \cite{VASSEUR2008323}.

Then, from \eqref{combined} and \eqref{note_that}, we get \eqref{combined1}.

\hfill\break
\uline{Step 2}
\hfill\break

Recall \eqref{X_def_calc}, which says $h(t)= s(t)-X(t)$. Choose $0<\epsilon<T-b$, and $R$ sufficiently large such that $-R<h(t)-\epsilon$ for all $t\in[0,T)$. 

We now show \eqref{nonlocal_dissipation_formula} for the case when $a=0$ and almost every $b\in(0,T)$. 

 Write \eqref{combined1} for the test function $\omega^0(t)\chi(x,t)$, where

\begin{align}
 \omega^0(t)\coloneqq
  \begin{cases}
   1 & \text{if } 0\leq t< b\\
   \frac{1}{\epsilon}(b-t)+1 & \text{if } b\leq t < b+\epsilon\\
   0 & \text{if } a+\epsilon \leq t,
  \end{cases}
\end{align}
and
\begin{align}\label{chi_for_combined1}
 \chi(x,t)\coloneqq
  \begin{cases}
   0 & \text{if } x<-2R\\
   \frac{1}{R}(x+2R) & \text{if } -2R\leq x < -R\\
   1 & \text{if } -R\leq x \leq h(t) -\epsilon\\
   -\frac{1}{\epsilon}(x-h(t)) & \text{if } h(t)-\epsilon<x\leq h(t)\\
   0 & \text{if } h(t)<x.
  \end{cases}
\end{align}

The function $\omega^0$ is modeled from \cite[p.~124]{dafermos_big_book}. The function $\chi$ is borrowed from \cite[p.~766]{Leger2011_original}. Note $\chi(h(t),t)=0$. We get,

\begin{equation}
\begin{aligned}\label{plug_in_test2}
&\int\limits_{0}^{b} \int\limits_{-2R}^{-R}\frac{1}{R} q(u(x,t);\bar{u}(x+X(t),t))\,dxdt 
\\
&+\int\limits_{-2R}^{-R}\frac{1}{R} (x+2R)\eta(u^0(x)|\bar{u}^0(x+X(0))\,dx -\int\limits_{b}^{b+\epsilon} \int\limits_{-2R}^{-R}\frac{1}{\epsilon R}(x+2R)\eta(u(x,t)|\bar{u}(x+X(t),t))\,dxdt
\\
&+\int\limits_{-R}^{h(0)}\eta(u^0(x)|\bar{u}^0(x+X(0)))\,dx-\int\limits_{b}^{b+\epsilon} \int\limits_{-R}^{h(t)}\frac{1}{\epsilon}\eta(u(x,t)|\bar{u}(x+X(t),t))\,dxdt
\\
&+\int\limits_{0}^{b} \int\limits_{h(t)-\epsilon}^{h(t)}\frac{1}{\epsilon}\dot{h}(t)\eta(u(x,t)|\bar{u}(x+X(t),t))\,dxdt-\int\limits_{0}^{b} \int\limits_{h(t)-\epsilon}^{h(t)}\frac{1}{\epsilon}q(u(x,t);\bar{u}(x+X(t),t))\,dxdt
\\
&-\int\limits_{h(0)-\epsilon}^{h(0)}\frac{1}{\epsilon}(x-h(0))\eta(u^0(x)|\bar{u}^0(x))\,dx
+\mbox{Error}(\epsilon)
\\
&\hspace{2in}\geq
\int\limits_{0}^{b} \int\limits_{-2R}^{-R}\frac{1}{R} (x+2R)\mbox{RHS}\,dxdt +\int\limits_{0}^{b} \int\limits_{-R}^{h(t)}\mbox{RHS}\,dxdt,
\end{aligned}
\end{equation}
where RHS represents everything being multiplied by $\phi$ in the integral on the right hand side of \eqref{combined1}. The term $\mbox{Error}(\epsilon)$ represents all terms which go to 0 as $\epsilon\to0$, for $R$ fixed.

We let $\epsilon\to0$ in \eqref{plug_in_test2}. We use the dominated convergence theorem, the Lebegue differentiation theorem, and recall that $u$ satisfies the strong trace property (\Cref{strong_trace_definition}). We can also drop the term 
\begin{align}
-\int\limits_{b}^{b+\epsilon} \int\limits_{-2R}^{-R}\frac{1}{\epsilon R}(x+2R)\eta(u(x,t)|\bar{u}(x+X(t),t))\,dxdt
\end{align}
because it is negative. This gives, 

\begin{equation}
\begin{aligned}\label{plug_in_test3}
&\int\limits_{0}^{b} \int\limits_{-2R}^{-R}\frac{1}{R} q(u(x,t);\bar{u}(x+X(t),t))\,dxdt 
+\int\limits_{-2R}^{-R}\frac{1}{R} (x+2R)\eta(u^0(x)|\bar{u}^0(x))\,dx
\\
&\hspace{1in}+\int\limits_{-R}^{h(0)}\eta(u^0(x)|\bar{u}^0(x))\,dx-\int\limits_{-R}^{h(b)}\eta(u(x,b)|\bar{u}(x+X(b),b))\,dx
\\
&+\int\limits_{0}^{b} \dot{h}(t)\eta(u(h(t)-,t)|\bar{u}((h(t)+X(t))-,t))\,dt-\int\limits_{0}^{b} q(u(h(t)-,t);\bar{u}((h(t)+X(t))-,t))\,dt
\\
&\hspace{1in}\geq
\int\limits_{0}^{b} \int\limits_{-2R}^{-R}\frac{1}{R} (x+2R)\mbox{RHS}\,dxdt +\int\limits_{0}^{b} \int\limits_{-R}^{h(t)}\mbox{RHS}\,dxdt,
\end{aligned}
\end{equation}
for almost every $b$.

To take the limit $R\to\infty$, we first note that RHS $\in L^1(\mathbb{R}\times[0,T))$. To see RHS $\in L^1(\mathbb{R}\times[0,T))$, remark that by virtue of the natural logarithm being locally square integrable, solutions $\bar{u}$ in the form \eqref{form_for_burgers_hilbert_paper} verify $\partial_x \bar{u} \in  L^2((\mathbb{R}\times[0,T))\setminus\{(x,t)|x=s(t)\})$. Further, recall that $\eta(a|b)$, $q(a;b)$, and $A(a|b)$  are locally quadratic in $a-b$ and that $u,\bar{u}\in L^2(\mathbb{R}\times[0,T))\cap L^\infty(\mathbb{R}\times[0,T))$. In particular, to show this term in the RHS is finite:
\begin{align}
\int\limits_{-\infty}^{\infty}\overbracket[.5pt][7pt]{\Bigg(
\partial_x \bigg|_{(x+X(t),t)} \hspace{-.45in}\eta'(\bar{u}(x,t))\Bigg)}^{L^2(\mathbb{R})}\overbracket[.5pt][7pt]{A(u(x,t)|\bar{u}(x+X(t),t))}^{L^2(\mathbb{R})}\,dx,
\end{align}
we use the indicated H\"older duality, recall $\partial_x \bar{u} \in  L^2((\mathbb{R}\times[0,T))\setminus\{(x,t)|x=s(t)\})$, and then use the basic $L^p$ interpolation inequality
\begin{align}
&\int\limits_{-\infty}^{\infty}\Big(u(x,t)-\bar{u}(x+X(t),t)\Big)^4\,dx\\
&=\int\limits_{-\infty}^{\infty}\Big(u(x,t)-\bar{u}(x+X(t),t)\Big)^2\Big(u(x,t)-\bar{u}(x+X(t),t)\Big)^2\,dx\\
&\leq\int\limits_{-\infty}^{\infty}\Big(u(x,t)-\bar{u}(x+X(t),t)\Big)^2\,dx\norm{\Big(u(\cdot,t)-\bar{u}(\cdot+X(t),t)\Big)^2}_{L^\infty(\mathbb{R})},
\end{align}
where we have used H\"older's inequality (recall that $A(a|b)$ is locally quadratic in $a-b$ so $(A(a|b))^2$ is quartic).

We let $R\to\infty$ in \eqref{plug_in_test3}. Recall the monotone convergence theorem, dominated convergence theorem, and that $u^0-\bar{u}^0\in L^2(\mathbb{R})$.  We get,

\begin{equation}
\begin{aligned}\label{plug_in_test4}
&\int\limits_{-\infty}^{h(0)}\eta(u^0(x)|\bar{u}^0(x+X(0)))\,dx-\int\limits_{-\infty}^{h(b)}\eta(u(x,b)|\bar{u}(x+X(b),b))\,dx
\\
&+\int\limits_{0}^{b}\dot{h}(t)\eta(u(h(t)-,t)|\bar{u}((h(t)+X(t))-,t))\,dt-\int\limits_{0}^{b} q(u(h(t)-,t);\bar{u}((h(t)+X(t))-,t))\,dt
\\
&\hspace{3in}\geq
 \int\limits_{0}^{b} \int\limits_{-\infty}^{h(t)}\mbox{RHS}\,dxdt.
\end{aligned}
\end{equation}
This gives the entropy dissipation on the left.

Similarly, we can calculate the entropy dissipation on the right,
\begin{equation}
\begin{aligned}\label{plug_in_test5}
&\int\limits_{h(0)}^{\infty}\eta(u^0(x)|\bar{u}^0(x+X(0)))\,dx-\int\limits_{h(b)}^{\infty}\eta(u(x,b)|\bar{u}(x+X(b),b))\,dx 
\\
&-\int\limits_{0}^{b}\dot{h}(t)\eta(u(h(t)+,t)|\bar{u}((h(t)+X(t))+,t))\,dt+\int\limits_{0}^{b} q(u(h(t)+,t);\bar{u}((h(t)+X(t))+,t))\,dt
\\
&\hspace{3in}\geq
 \int\limits_{0}^{b} \int\limits_{h(t)}^{\infty}\mbox{RHS}\,dxdt.
\end{aligned}
\end{equation}

Recall that $h(t)=s(t)-X(t)$. Then, we add \eqref{plug_in_test4} and \eqref{plug_in_test5} to get \eqref{nonlocal_dissipation_formula} for $a=0$ and almost every $b\in(0,T)$.

On the other hand, for almost every $a,b\in(0,T)$ with $a<b$ we write \eqref{combined1} for the test function $\omega(t)\chi(x,t)$, where $\chi(x,t)$ is as above in \eqref{chi_for_combined1} but instead of the $\omega^0(t)$ used above we use 
\begin{align}
 \omega(t)\coloneqq
  \begin{cases}
     0 & \text{if } 0\leq t < a\\
   \frac{1}{\epsilon}(t-a) & \text{if } a\leq t < a+\epsilon\\
   1 & \text{if }a+\epsilon \leq t< b\\
   \frac{1}{\epsilon}(b-t)+1 & \text{if } b\leq t < b+\epsilon\\
   0 & \text{if } b+\epsilon \leq t.
  \end{cases}
\end{align}

Writing \eqref{combined1} for the test function $\omega(t)\chi(x,t)$ and following a very similar argument to the one above for $a=0$ completes the proof of \eqref{nonlocal_dissipation_formula}.

\end{proof}

\begin{lemma}[Structural lemma from \cite{serre_vasseur}]
\label{entropy_dissipation_rewrite_lemma_scalar}
Let $u_+,u_-,\bar{u}_+,\bar{u}_-\in\mathbb{R}$ satisfy $u_-\geq u_+$ and $\bar{u}_-\geq\bar{u}_+$. 

Define
\begin{align}\label{RH_relation}
\sigma(u_+,u_-) \coloneqq
  \begin{cases}
   \frac{A(u_+)-A(u_-)}{u_+-u_-} & \text{if } u_+\neq u_- \\
   A'(u_+)       & \text{if } u_+=u_-.
  \end{cases}
\end{align}
To simplify notation, we write $\sigma=\sigma(u_+,u_-)$.

Define the following real-valued map $B$ on the set of  intervals $I\subset\mathbb{R}$:

\begin{align}
B(I)\coloneqq \int\limits_{I}\Big[ \eta''(u)\big[(A(u)-\sigma u)-(A(u)-\sigma u)_\pm\big]\Big]\,du,
\end{align}
where $(A(u)-\sigma u)_\pm$ denotes that
\begin{align}\label{pm_def1}
A(u_+)-\sigma u_+=A(u_-)-\sigma u_-,
\end{align}
due to the Rankin-Hugoniot relation \eqref{RH_relation}. Note $B(\varnothing)=0$.
Let $I$ and $J$ be disjoint intervals such that
\begin{align}
I\cup J=((u_+,u_-)\cup(\bar{u}_+,\bar{u}_-))\setminus((u_+,u_-)\cap(\bar{u}_+,\bar{u}_-)).
\end{align}
We allow for $I$ and/or $J=\varnothing$.
Further, define $\epsilon(I)$ to be $+1$ if $I\subset (u_+,u_-)$ and $-1$ otherwise.

Then,
\begin{align}\label{entropy_dissipation_rewrite_scalar}
q(u_+;\bar{u}_+)-q(u_-;\bar{u}_-)-\sigma(\eta(u_+|\bar{u}_+)-\eta(u_-|\bar{u}_-))=\epsilon(I)B(I)+\epsilon(J)B(J).
\end{align}
\end{lemma}

\begin{proof}
This proof is from \cite[p.~9-10]{serre_vasseur}. In the article \cite{serre_vasseur}, the authors develop a condition for the general systems case which they label equation number 7 \cite[p.~4]{serre_vasseur}. They claim to use this condition in the argument for the scalar conservation laws which we are using here to prove \Cref{entropy_dissipation_rewrite_lemma_scalar}. In fact, they do not need condition number 7 for the scalar case and their proof goes through unchanged without the condition. We do not use the condition here. The authors in \cite{serre_vasseur} also restrict themselves to scalar solutions given by Kruzhkov's theory of scalar conservation laws (see \cite{MR0267257}), but again their proofs go through unchanged without this assumption.

We use the following notation, inspired by \cite{serre_vasseur}: If $F$ is a function of $u$, then we define
\begin{align}
[F]\coloneqq F(\bar{u}_+)-F(\bar{u}_-) \hspace{.3in}\mbox{and}\hspace{.3in} F_\pm\coloneqq F(u_\pm).
\end{align}

Assume for now that $u_+\neq u_-$.

Denote
\begin{align}
D\coloneqq q(u_+;\bar{u}_+)-q(u_-;\bar{u}_-)-\sigma(\eta(u_+|\bar{u}_+)-\eta(u_-|\bar{u}_-)).
\end{align}

From Rankine-Hugoniot \eqref{RH_relation} (as remarked in \cite[p.~5]{serre_vasseur}),
\begin{align}
D=[\eta'(u)A(u)-q(u)]-\sigma[\eta'(u)u-\eta]+q_+-q_- -\sigma(\eta_+-\eta_-)-[\eta'](A(u)-\sigma u)_\pm,
\end{align}
where $(A(u)-\sigma u)_\pm$ denotes the fact that 
\begin{align}
A(u_+)-\sigma u_+=A(u_-)-\sigma u_-
\end{align}
due to the Rankine-Hugoniot relation \eqref{RH_relation}.

The fundamental theorem of calculus and integration by parts yield,
\begin{align}\label{entropy_dissipation_rewrite_scalar_almost}
D=\int\limits_{u_+}^{u_-}\Big[ \eta''(u)\big[(A(u)-\sigma u)-(A(u)-\sigma u)_\pm\big]\Big]\,du\\
-\int\limits_{\bar{u}_+}^{\bar{u}_-}\Big[ \eta''(u)\big[(A(u)-\sigma u)-(A(u)-\sigma u)_\pm\big]\Big]\,du.
\end{align}
This proves \eqref{entropy_dissipation_rewrite_scalar} for $u_+\neq u_-$.

We prove \eqref{entropy_dissipation_rewrite_scalar} for $u_+=u_-$ by using \eqref{entropy_dissipation_rewrite_scalar_almost} and continuity, in particular the continuity of the $\sigma$ function (see \Cref{H_cont}).  This proves the lemma.

\end{proof}

\begin{lemma}[Structural lemma on entropic shocks from \cite{2017arXiv170905610K}]\label{entropic_shock_lemma_scalar_lax}

Assume the system \eqref{system} has a strictly convex flux $A\in C^2(\mathbb{R})$, and is endowed with a strictly convex entropy $\eta\in C^2(\mathbb{R})$, with an associated entropy flux $q$. Let $u_L,u_R,\sigma\in\mathbb{R}$ verify
\begin{align}
A(u_L)-A(u_R)=\sigma(u_L-u_R).
\end{align}
Then,
\begin{align}
q(u_R)-q(u_L)\leq \sigma (\eta(u_R)-\eta(u_L))
\end{align}
if and only if $u_L\geq u_R$. I.e., the shock $(u_L,u_R,\sigma)$ is entropic for the entropy $\eta$ if and only if $u_L\geq u_R$.
\end{lemma}
\begin{remark}
For the system \eqref{system} with a strictly convex flux $A$, the ``physical'' condition on a shock with left-hand state $u_L$ and right-hand state $u_R$ is to require that $u_L>u_R$. This is the \emph{Lax entropy condition} for the scalar systems. \Cref{entropic_shock_lemma_scalar_lax} says that the shock $(u_L,u_R,\sigma)$ being entropic for the the entropy $\eta$ is equivalent to satisfying the Lax entropy condition.
\end{remark}
We do not prove \Cref{entropic_shock_lemma_scalar_lax} here. For a proof, see \cite[p.~13]{2017arXiv170905610K}.

\begin{lemma}[Structural lemma on approximate limits in time from \cite{2017arXiv170905610K}]\label{left_right_ap_limits}

Fix $T>0$.

Assume that $u,\bar{u}$ are weak solutions to \eqref{system}. Assume that $u$ is entropic for the entropy $\eta\in C^3(\mathbb{R})$. Assume that $\bar{u}$ is in the form \eqref{form_for_burgers_hilbert_paper}.  Assume that $u \in L^2(\mathbb{R}\times[0,T))\cap L^\infty(\mathbb{R}\times[0,T))$.  Further, assume that $u^0-\bar{u}^0\in L^2(\mathbb{R})$. Assume also that $u$ verifies the strong trace property (\Cref{strong_trace_definition}).

Let  $h:[0,T)\to\mathbb{R}$ be a Lipschitz continuous map.

Define
\begin{align}
X(t)\coloneqq s(t)-h(t),\label{X_def_calc}
\end{align}
where $s$ is as in \eqref{form_for_burgers_hilbert_paper}.

Then the approximate right- and left-hand limits
\begin{align}\label{ap_right_left_limits_exist_44}
&{\rm ap}\,\lim_{t\to {t_0}^{\pm}}\int\limits_{-\infty}^{\infty}\eta (u(x,t)|\bar{u}(x+X(t),t))\,dx
\end{align}
exist for all $t_0\in(0,\infty)$ and verify
\begin{align}
\label{left_and_right_limits_order}
&{\rm ap}\,\lim_{t\to {t_0}^{-}}\int\limits_{-\infty}^{\infty}\eta (u(x,t)|\bar{u}(x+X(t),t))\,dx\geq {\rm ap}\,\lim_{t\to {t_0}^{+}}\int\limits_{-\infty}^{\infty}\eta (u(x,t)|\bar{u}(x+X(t),t))\,dx.
\end{align}
The approximate right-hand limit also exists for $t_0=0$ and verifies 
\begin{align}
\label{right_limit_at_zero}
\int\limits_{-\infty}^{\infty}\eta (u^0(x)|\bar{u}^0(x+X(0)))\,dx\geq {\rm ap}\,\lim_{t\to {0}^{+}}\int\limits_{-\infty}^{\infty}\eta (u(x,t)|\bar{u}(x+X(t),t))\,dx.
\end{align}
\end{lemma}

The proof of \Cref{left_right_ap_limits} is very similar to the proof of Lemma 2.4 in \cite[p.~11]{2017arXiv170905610K}. For completeness, a proof of \Cref{left_right_ap_limits} is in the appendix (\Cref{approx_limits_appendix}).

\begin{lemma}\label{H_cont}
Let $f: \mathbb{R}\to\mathbb{R}$ be in $C^3(\mathbb{R})$.

Let 
\[
H(x,y) \coloneqq
  \begin{cases}
   \frac{f(x)-f(y)}{x-y} & \text{if } x\neq y \\
   f'(x)       & \text{if } x=y.
  \end{cases}
\]

Then $H\in C^1(\mathbb{R}^2)$.
\end{lemma}
\begin{proof}
We show the first partials of $H$ exist and are continuous. 

Note that $H(x,y)=H(y,x)$ so we only have to show that $H_{x}$ is continuous.

For $(x,y)\in \mathbb{R}^2$, $x\neq y$, 
 \begin{align}
H_{x}(x,y)=\frac{f'(x)}{x-y}-\frac{f(x)-f(y)}{(x-y)^2}.
\end{align}

We calculate $H_{x}$ at points $(x,x)\in \mathbb{R}^2$:
\begin{align}
\label{H_limit}
H_{x}(x,x)&=\lim_{h\to0} \frac{\frac{f(x+h)-f(x)}{h}-f'(x)}{h}.
\end{align}
We use Taylor's theorem to evaluate the limit \eqref{H_limit}. Write
\begin{align}
\label{taylor_1}
f(x+h)=f(x)+f'(x)h+\frac{f''(x)}{2}h^2+R(h)h^2,
\end{align}
where $\lim_{h\to0} R(h)=0$.

From \eqref{taylor_1}, we have 
\begin{align}
H_{x}(x,x)&=\lim_{h\to0}\Bigg[ \frac{f''(x)}{2}+R(h)\Bigg]=\frac{f''(x)}{2}.
\end{align}

Next, we show that
\begin{align}\label{want_to_show_H}
\lim_{(x,y)\to (x_0,x_0)}H_{x}(x,y)= \lim_{(x,y)\to (x_0,x_0)} \frac{f'(x)}{x-y}-\frac{f(x)-f(y)}{(x-y)^2} = \frac{f''(x_0)}{2}.
\end{align}

We use Taylor's theorem again. Write
\begin{align}
\label{taylor_2}
f(y)=f(x)+f'(x)(y-x)+\frac{f''(x)}{2}(y-x)^2+Q(x,y).
\end{align}

The remainder term $Q(x,y)$ satisfies 
\begin{align}
Q(x,y)=\frac{f^{(3)}(\xi(x,y))}{3!}(y-x)^3,
\end{align}
where $\xi(x,y)\in \mathbb{R}$ is between $y$ and $x$.

Then, we substitute the formula for $f(y)$ (from \eqref{taylor_2}) into \eqref{want_to_show_H}. This yields,
\begin{align}\label{H_piece_scalar}
\frac{f'(x)}{x-y}-\frac{f(x)-f(y)}{(x-y)^2}=\frac{f''(x)}{2}+\frac{Q(x,y)}{(x-y)^2}.
\end{align}

For $(x,y)$ close to $(x_0,x_0)$, $\xi(x,y)$ is close to $x_0$. Thus, $f^{(3)}(\xi(x,y))$ is bounded for $(x,y)$ in a neighborhood of $(x_0,x_0)$.

Thus,
\begin{align}\label{remainder_behavior_H}
\abs{\frac{Q(x,y)}{(x-y)^2}}=\abs{\frac{f^{(3)}(\xi(x,y))}{3!}}\abs{y-x}\to 0
\end{align}
as $(x,y)\to(x_0,x_0)$.

Putting together \eqref{want_to_show_H}, \eqref{H_piece_scalar}, and \eqref{remainder_behavior_H}, we get
\begin{align}
\lim_{(x,y)\to (x_0,x_0)}H_{x}(x,y)= \frac{f''(x_0)}{2}.
\end{align}
This shows that $H_x$ and $H_y$ are continuous on $\mathbb{R}^2$. We conclude  $H\in C^1(\mathbb{R}^2)$.
\end{proof}

\section{Structural lemma on the negativity of entropy dissipation}\label{neg_entrop_sec}

The following Lemma says that if a discontinuity in the second slot of relative entropy is being artificially shifted at the speed of a discontinuity in the first slot of the relative entropy,  we get  entropy dissipation proportional to the square of the shift. 

\begin{lemma}[Structural lemma on the negativity of entropy dissipation]\label{negative_entropy_diss_scalar_lemma}
Fix $\delta, B>0$. Let $u_+,u_-,\bar{u}_+,\bar{u}_-\in\mathbb{R}$ satisfy $u_-\geq u_+$, $\bar{u}_- - \bar{u}_+\geq\delta$ and 
\begin{align}\label{boundedness_scalar1}
\abs{u_+},\abs{u_-},\abs{\bar{u}_+},\abs{\bar{u}_-}\leq B.
\end{align}

Define
\begin{align}\label{RH_relation_1}
\sigma(u_+,u_-) \coloneqq
  \begin{cases}
   \frac{A(u_+)-A(u_-)}{u_+-u_-} & \text{if } u_+\neq u_- \\
   A'(u_+)       & \text{if } u_+=u_-.
  \end{cases}
\end{align}

Then,
\begin{equation}
\begin{aligned}\label{entropy_dissipation_room_scalar}
q(u_+;\bar{u}_+)-q(u_-;\bar{u}_-)-\sigma(u_+,u_-)(\eta(u_+|\bar{u}_+)-\eta(u_-|\bar{u}_-))
\\ \leq -c\big((u_+-\bar{u}_+)^2+(u_--\bar{u}_-)^2\big),
\end{aligned}
\end{equation}
where $c>0$ is a constant that depends on $B$ and $\delta$.
\end{lemma}
\begin{proof}
Denote 
\begin{align}
D\coloneqq q(u_+;\bar{u}_+)-q(u_-;\bar{u}_-)-\sigma(u_+,u_-)(\eta(u_+|\bar{u}_+)-\eta(u_-|\bar{u}_-)).
\end{align}

Note that $D$ is continuous.

We use \Cref{entropy_dissipation_rewrite_lemma_scalar}. 

Note that $\eta''>0$. Note also that the map $\Gamma$:
\begin{align}\label{Gamma_map}
u\mapsto \big[(A(u)-\sigma u)-(A(u)-\sigma u)_\pm\big]
\end{align}
is \emph{strictly} convex and satisfies $\Gamma(u_+)=\Gamma(u_-)=0$ due to \eqref{pm_def1}. Thus for the two intervals $I$ and $J$ in \Cref{entropy_dissipation_rewrite_lemma_scalar}, we have $\epsilon(I)B(I)\leq 0$ and $\epsilon(J)B(J)\leq 0$.

In particular, for any  
\begin{align}\label{case1_scalar}
\begin{cases}
\mbox{$u_+,u_-,\bar{u}_+,\bar{u}_-\in\mathbb{R}$ with $0\leq u_- - u_+\leq \frac{\delta}{2}$, $\bar{u}_- - \bar{u}_+\geq\delta$} \\ \mbox{and $\max\{\abs{u_+},\abs{u_-},\abs{\bar{u}_+},\abs{\bar{u}_-}\}\leq B$,}
\end{cases}
\end{align}
we must have $D$ \emph{strictly} less than zero because the measure of the interval $(\bar{u}_+,\bar{u}_-)$ is at least $\delta$, while the measure of the interval $(u_+,u_-)$ is less than or equal to $\frac{\delta}{2}$.  The set of such $u_+,u_-,\bar{u}_+,\bar{u}_-$ is closed and bounded. Thus, we know the infimum of the continuous function $D$ over this set \eqref{case1_scalar} is also strictly negative,
\begin{align}
\inf D < 0.
\end{align}
Due to \eqref{boundedness_scalar1}, we can make $c$ small enough such that  \eqref{entropy_dissipation_room_scalar} holds for the set \eqref{case1_scalar}. 

Consider now the set of all
\begin{align}\label{case2_scalar}
\begin{cases}
\text{$u_+,u_-,\bar{u}_+,\bar{u}_-\in\mathbb{R}$ with $u_- - u_+\geq \frac{\delta}{2}$, $\bar{u}_- - \bar{u}_+\geq\delta$,}\\ 
\text{$\max\{\abs{u_+},\abs{u_-},\abs{\bar{u}_+},\abs{\bar{u}_-}\}
\leq B$}\\
\text{and at least one of the following is true:}\\
\text{(i) $\abs{u_+-\bar{u}_+}\geq \frac{\delta}{2},$}\\
\text{(ii) $\abs{u_- -\bar{u}_-}\geq \frac{\delta}{2}$.}
\end{cases}
\end{align}

Recall \Cref{entropy_dissipation_rewrite_lemma_scalar} and that the map $\Gamma$ (see \eqref{Gamma_map}) is \emph{strictly} convex. Then for $u_+,u_-,\bar{u}_+,\bar{u}_-$ in the set \eqref{case2_scalar}, by inspection of the possible cases,
\begin{itemize}
\item
$(u_+,u_-)\cap(\bar{u}_+,\bar{u}_-)=\varnothing$
\item
$\bar{u}_+\leq u_+ \leq \bar{u}_- \leq u_-$

\item
$u_+\leq \bar{u}_+\leq\bar{u}_-\leq u_-$

\item
$\bar{u}_+\leq u_+ \leq u_- \leq \bar{u}_-$

\item
$u_+\leq \bar{u}_+ \leq {u}_- \leq \bar{u}_-$,
\end{itemize}

it is clear that $D$ will always be \emph{strictly} less than zero. 

The set \eqref{case2_scalar} is closed and bounded. To show it is closed, consider a sequence $\{(u_{+,n},u_{-,n},\bar{u}_{+,n},\bar{u}_{-,n})\}_{n\in\mathbb{N}}$ in the set \eqref{case2_scalar} converging to a point $(u_+,u_-,\bar{u}_+,\bar{u}_-)$. If both
\begin{align}
\mbox{(i) } \abs{u_+-\bar{u}_+}< \frac{\delta}{2},\\
\mbox{(ii) } \abs{u_- -\bar{u}_-}< \frac{\delta}{2},
\end{align}
are true, then there exists some $N\in\mathbb{N}$ such that
\begin{align}
\mbox{(i) } \abs{u_{+,N}-\bar{u}_{+,N}}< \frac{\delta}{2},\\
\mbox{(ii) } \abs{u_{-,N} -\bar{u}_{-,N}}< \frac{\delta}{2},
\end{align}
which is a contradiction. It is obvious that the set \eqref{case2_scalar} is closed with respect to the rest of its properties. Thus the set \eqref{case2_scalar} is closed and bounded, and again we conclude the infimum of the continuous function $D$ on this set must be strictly negative. Thus, due to \eqref{boundedness_scalar1} we can make $c$ sufficiently small such that \eqref{entropy_dissipation_room_scalar} holds.

We now make note of a few rudimentary facts.

Due to the strict convexity of $\Gamma$ and $\Gamma(u_+)=\Gamma(u_-)=0$, there are positive constants $c^*$ and $c^{**}$ such that for $u\in[-B,B]$,
\begin{align}\label{trapped_parabs}
c^*(u-u_+)(u-u_-)\leq \Gamma(u) \leq c^{**}(u-u_+)(u-u_-).
\end{align}

We define $F:\mathbb{R}^2\to\mathbb{R}$,
\begin{align}\label{F_def}
F(a,b)\coloneqq \int\limits_a^b &(u-u_+)(u-u_-)\,du,
\end{align}
for $a,b\in\mathbb{R}$.

Note that,
\begin{equation}
\begin{aligned}\label{simple_int}
F(a,b)&=\int\limits_a^b (u-u_+)(u-u_-)\,du \\
 &=(b-u_+)^2\Big(\frac{b-u_+}{3}+\frac{u_+-u_-}{2}\big)-(a-u_+)^2\Big(\frac{a-u_+}{3}+\frac{u_+-u_-}{2}\big),
\\&\hspace{-1in}\mbox{and from the symmetry $u_+\leftrightarrow u_-$,}\\
&=(b-u_-)^2\Big(\frac{b-u_-}{3}+\frac{u_--u_+}{2}\big)-(a-u_-)^2\Big(\frac{a-u_-}{3}+\frac{u_--u_+}{2}\big).
\end{aligned}
\end{equation}

Consider now the final set,
\begin{align}\label{case3_scalar}
\begin{cases}
\text{$u_+,u_-,\bar{u}_+,\bar{u}_-\in\mathbb{R}$ with $u_- - u_+\geq \frac{\delta}{2}$, $\bar{u}_- - \bar{u}_+\geq\delta$,}\\ 
\text{$\max\{\abs{u_+},\abs{u_-},\abs{\bar{u}_+},\abs{\bar{u}_-}\}
\leq B$,}\\
\text{and we have both:}\\
\text{(i) $\abs{u_+-\bar{u}_+}< \frac{\delta}{2},$}\\
\text{\hspace{.5in} and}\\
\text{(ii) $\abs{u_- -\bar{u}_-}< \frac{\delta}{2}$.}
\end{cases}
\end{align}

For each $(u_+,u_-,\bar{u}_+,\bar{u}_-)$ in the set \eqref{case3_scalar}, we show \eqref{entropy_dissipation_room_scalar} by analyzing each of the possible cases:

\begin{itemize}
\item
$(u_+,u_-)\cap(\bar{u}_+,\bar{u}_-)=\varnothing$
\item
$\bar{u}_+\leq u_+ \leq \bar{u}_- \leq u_-$

\item
$\bar{u}_+\leq u_+ \leq u_- \leq \bar{u}_-$

\item
$u_+\leq \bar{u}_+ \leq \bar{u}_- \leq u_-$

\item
$u_+\leq \bar{u}_+ \leq {u}_- \leq \bar{u}_-$
\end{itemize}

\emph{Case:}  $(u_+,u_-)\cap(\bar{u}_+,\bar{u}_-)=\varnothing$

Rudimentary geometric arguments show this case is not possible (for a fixed $u_+$ and $u_-$, consider where on the real line would $\bar{u}_+$ be?).

\emph{Case:}  $\bar{u}_+\leq u_+ \leq \bar{u}_- \leq u_-$

From \Cref{entropy_dissipation_rewrite_lemma_scalar},  \eqref{trapped_parabs}, and \eqref{F_def} we get,
\begin{align}
D \leq -c^* [\inf\eta''] F(\bar{u}_+,u_+) +c^{**}[\inf\eta'']F(\bar{u}_-,u_-)\\
\leq  -\frac{1}{4}\delta c^* [\inf\eta''] (\bar{u}_+-u_+)^2 -\frac{1}{12}\delta c^{**}[\inf\eta''] (\bar{u}_- - u_-)^2,
\end{align}
from \eqref{simple_int} and \eqref{case3_scalar}.

This shows \eqref{entropy_dissipation_room_scalar}.

\emph{Case:}  $\bar{u}_+\leq u_+ \leq u_- \leq \bar{u}_-$

From \Cref{entropy_dissipation_rewrite_lemma_scalar},  \eqref{trapped_parabs}, and \eqref{F_def} we get,
\begin{align}
D \leq -c^* [\inf\eta''] F(\bar{u}_+,u_+) -c^{*}[\inf\eta'']F(u_-,\bar{u}_-)\\
\leq  -\frac{1}{4}\delta c^* [\inf\eta''] (\bar{u}_+-u_+)^2 -\frac{1}{4}\delta c^{*}[\inf\eta''] (\bar{u}_- - u_-)^2,
\end{align}
from \eqref{simple_int} and \eqref{case3_scalar}.
This shows \eqref{entropy_dissipation_room_scalar}.

\emph{Case:}  $u_+\leq \bar{u}_+ \leq \bar{u}_- \leq u_-$

From \Cref{entropy_dissipation_rewrite_lemma_scalar},  \eqref{trapped_parabs}, and \eqref{F_def} we get,
\begin{align}
D \leq c^{**} [\inf\eta''] F(u_+,\bar{u}_+) +c^{**}[\inf\eta'']F(\bar{u}_-,u_-)\\
\leq  -\frac{1}{12}\delta c^{**} [\inf\eta''] (\bar{u}_+-u_+)^2 -\frac{1}{12}\delta c^{**}[\inf\eta''] (\bar{u}_- - u_-)^2,
\end{align}
from \eqref{simple_int} and \eqref{case3_scalar}.
This shows \eqref{entropy_dissipation_room_scalar}.

\emph{Case:}  $u_+\leq \bar{u}_+ \leq {u}_- \leq \bar{u}_-$

From \Cref{entropy_dissipation_rewrite_lemma_scalar},  \eqref{trapped_parabs}, and \eqref{F_def} we get,
\begin{align}
D \leq c^{**} [\inf\eta''] F(u_+,\bar{u}_+) -c^{*}[\inf\eta'']F(u_-,\bar{u}_-)\\
\leq  -\frac{1}{12}\delta c^{**} [\inf\eta''] (\bar{u}_+-u_+)^2 -\frac{1}{4}\delta c^{*}[\inf\eta''] (\bar{u}_- - u_-)^2,
\end{align}
from \eqref{simple_int} and \eqref{case3_scalar}.
This shows \eqref{entropy_dissipation_room_scalar}.

This completes the proof of \Cref{negative_entropy_diss_scalar_lemma}.
\end{proof}

\section{Proof of the main theorem (\Cref{main_stability_result_statement_scalar})}\label{main_proof_sec}

To prove the main theorem for solutions defined on the finite time interval $[0,T)$ (\Cref{main_stability_result_statement_scalar}), we first prove it for any time interval of length $\frac{1}{C}$ for a uniform constant $C$ (the main proposition -- \Cref{main_stability_result_statement_scalar_unit_time}). Then we use induction (\Cref{induction_section}) to extend the time  interval to $[0,T)$. The constants will depend on the magnitude of the time $T$, but since $T$ is always assumed to be finite in this article (and in \cite{MR3605552}), we are okay.

\begin{proposition}[Main proposition -- $L^2$ stability for entropic piecewise-Lipschitz solutions to scalar balance laws for uniformly small time]\label{main_stability_result_statement_scalar_unit_time}

 Fix $T>0$.

Consider $u,\bar{u}$ weak solutions to \eqref{system}. Assume  $u \in L^2(\mathbb{R}\times[0,T))\cap L^\infty(\mathbb{R}\times[0,T))$ verifies the strong trace property (\Cref{strong_trace_definition}) and is entropic for the strictly convex entropy $\eta\in C^3(\mathbb{R})$, where $\eta$ and $G$ verify \eqref{requirements_G_piecewise_scalar}. Further, assume $\bar{u}$ is in the form \eqref{form_for_burgers_hilbert_paper}. Furthermore, assume that $u^0-\bar{u}^0\in L^2(\mathbb{R})$.

Assume also that there exists $\delta>0$ such that 
\begin{align}\label{delta_room_scalar_unit_time}
\bar{u}(s(t)-,t)-\bar{u}(s(t)+,t)>\delta
\end{align}
for all $t\in[0,T)$.

Then, there exists a constant $C>0$ and a Lipschitz continuous function $X:[0,T)\to\mathbb{R}$ with $X(0)=0$ and such that for $a,b\in[0,T)$ with $0\leq b-a\leq \frac{1}{C}$,
\begin{equation}
\begin{aligned}\label{main_global_stability_result_scalar_unit_time}
&{\rm ap}\,\lim_{t\to {b}^{-}}\int\limits_{\mathbb{R}}\eta(u(x,t)|\bar{u}(x+X(t),t))\,dx\\
\leq
C \Bigg[\Bigg(&{\rm ap}\,\lim_{t\to {a}^{+}}\int\limits_{\mathbb{R}}\eta(u(x,t)|\bar{u}(x+X(t),t))\,dx\Bigg)^{1/4}
\\
+
\Bigg(&{\rm ap}\,\lim_{t\to {a}^{+}}\int\limits_{\mathbb{R}}\eta(u(x,t)|\bar{u}(x+X(t),t))\,dx\Bigg)^3\Bigg] e^{C(b-a)+
C\int\limits_a^{b}\norm{\partial_ x\bar{u}(\cdot,t)}_{L^2(\mathbb{R})}^2\,dt},
\end{aligned}
\end{equation}
where ${\rm ap}\,\lim$ denotes the approximate limit.

Moreover, 
\begin{align}
X(t)=s(t)-h(t),
\end{align}
where $h(t)$ is a generalized characteristic of $u$.

The constant $C$ depends on $\delta$ and $T$.
\end{proposition}
\begin{proof}

Note that because $\bar{u}$ is in the form \eqref{form_for_burgers_hilbert_paper}, it is smooth  on $\{(x,t)\in\mathbb{R}\times[0,T) | x<s(t)\}$ and on $\{(x,t)\in\mathbb{R}\times[0,T) | x>s(t)\}$, where $s:[0,T)\to\mathbb{R}$ is a Lipschitz function . Further, $\partial_x \bar{u} \in  L^2((\mathbb{R}\times[0,T))\setminus\{(x,t)|x=s(t)\})$ and $\bar{u}\in L^2(\mathbb{R}\times[0,T))\cap L^\infty(\mathbb{R}\times[0,T))$.

We let $C$ denote a generic constant, in particular depending on $\delta$.

\uline{Step 1}
We solve the following ODE in the Filippov sense:

\begin{align}
  \begin{cases}\label{ODE_scalar_shift_h}
   \dot{h}(t)=A'(u(h(t),t))\\
   h(0)=s(0).
  \end{cases}
\end{align}

We use the following lemma.
\begin{lemma}[Existence of Filippov flows]\label{Filippov_existence_scalar}
Let $V(u,t):\mathbb{R} \times [0,\infty)\to\mathbb{R}$ be bounded on $\mathbb{R} \times [0,\infty)$, continuous in $u$, and measurable in $t$. Let $u$ be a bounded, weak solution to \eqref{system}, entropic for the entropy $\eta$. Assume also that $u$ verifies the strong trace property (\Cref{strong_trace_definition}). Let $x_0\in\mathbb{R}$. Then we can solve 
\begin{align}
  \begin{cases}\label{ODE_scalar}
   \dot{h}(t)=V(u(h(t),t),t)\\
   h(0)=x_0,
  \end{cases}
\end{align}
in the Filippov sense. Which is to say, there exists a Lipschitz function $h:[0,\infty)\to\mathbb{R}$ such that
\begin{align}
\mbox{Lip}[h]\leq \norm{V}_{L^\infty},\label{fact1_scalar}\\
h(0)=x_0,\label{fact2_scalar}
\shortintertext{and}
\dot{h}(t)\in I[V(u_+,t),V(u_-,t)],\label{fact3_scalar}
\end{align}
for almost every $t$. We denote $u_\pm\coloneqq u(h(t)\pm,t)$. We use $I[a,b]$ to denote the closed interval with endpoints $a$ and $b$.

Furthermore, for almost every $t$,
\begin{align}
f(u_+)-f(u_-)=\dot{h}(u_+-u_-),\label{fact4_scalar}\\
q(u_+)-q(u_-)\leq\dot{h}(\eta(u_+)-\eta(u_-)).\label{fact5_scalar}
\end{align}
I.e., for almost every $t$, either $(u_+,u_-,\dot{h})$ is a shock entropic for $\eta$ or $u_+=u_-$.
\end{lemma}

The proof of \eqref{fact1_scalar}, \eqref{fact2_scalar}, and \eqref{fact3_scalar} follows closely the proof of Proposition 1 in \cite{Leger2011} and the proof of Lemma 3.5 in \cite{2017arXiv170905610K}. See \Cref{Filippov_existence_scalar_proof_section} for a proof of  \eqref{fact1_scalar}, \eqref{fact2_scalar}, and \eqref{fact3_scalar}.

It is well known that \eqref{fact4_scalar} and \eqref{fact5_scalar} are true generally for any Lipschitz continuous function $h:[0,\infty)\to\mathbb{R}$ when $u$ is BV. When instead $u$ is only known to have strong traces (\Cref{strong_trace_definition}), then \eqref{fact4_scalar} and \eqref{fact5_scalar} are stated in Lemma 6 in \cite{Leger2011}. 
For proofs of  \eqref{fact4_scalar} and \eqref{fact5_scalar}, see the appendix in \cite{Leger2011}. We do not prove the properties \eqref{fact4_scalar} and \eqref{fact5_scalar} here.

Remark that the function $h$ is a generalized characteristic for the solution $u$ (see \cite[Chapter 10]{dafermos_big_book}).

\uline{Step 2}

Denote 
\begin{equation}
\begin{aligned}
u_\pm\coloneqq u(h(t)\pm,t),\\
\bar{u}_\pm\coloneqq \bar{u}(s(t)\pm,t).
\end{aligned}
\end{equation}

Then \Cref{entropic_shock_lemma_scalar_lax}, \eqref{fact4_scalar}, and \eqref{fact5_scalar} imply that $u_-\geq u+$. Further, \eqref{delta_room_scalar} implies that $\bar{u}_->\bar{u}_+$. From \eqref{ODE_scalar_shift_h}, \eqref{fact3_scalar}, \eqref{fact4_scalar}, and \eqref{RH_relation_1}, we see $\dot{h}(t)=\sigma(u_+,u_-)$. Thus, from \Cref{negative_entropy_diss_scalar_lemma} we have
\begin{align}\label{entropy_dissipation_room_scalar_in_use}
q(u_+;\bar{u}_+)-q(u_-;\bar{u}_-)-\dot{h}(t)(\eta(u_+|\bar{u}_+)-\eta(u_-|\bar{u}_-)) \leq -c\big((u_+-\bar{u}_+)^2+(u_--\bar{u}_-)^2\big),
\end{align}
where $c$ is the constant from the right hand side of\eqref{entropy_dissipation_room_scalar} and it depends on $\norm{u_\pm(t)}_{L^\infty([0,T))}$, $\norm{\bar{u}(s(t)\pm,t)}_{L^\infty([0,T))}$  and $\delta$.

Because $A\in C^3 (\mathbb{R})$, the function $\sigma:\mathbb{R}^2\to\mathbb{R}$, defined as
\begin{align}\label{RH_relation_2}
\sigma(v,w) \coloneqq
  \begin{cases}
   \frac{A(v)-A(w)}{v-w} & \text{if } v\neq w \\
   A'(v)       & \text{if } v=w,
  \end{cases}
\end{align}
is in $C^1(\mathbb{R}^2)$ by \Cref{H_cont}. Thus, from Taylor's theorem,
\begin{align}\label{control_on_difference_sigmas_scalar}
\abs{\sigma(\bar{u}_+,\bar{u}_-)-\sigma(u_+,u_-)}
\leq \sup\abs{\partial_v \sigma(\xi)}\Big(\abs{\bar{u}_+-u_+)}+\abs{\bar{u}_--u_-}\Big),
\end{align}
where the supremum of $\abs{\partial_v \sigma(\xi)}$ runs over the set of $\xi$ such that $\abs{\xi}\leq\max\{\norm{u}_{L^\infty},\norm{\bar{u}}_{L^\infty}\}$.
Note that we have used the symmetry  of $\sigma$. In particular, $\partial_v \sigma(\xi)=\partial_w \sigma(\xi)$.

From \eqref{control_on_difference_sigmas_scalar}, we get
\begin{align}\label{control_on_difference_sigmas_scalar_2}
\abs{\sigma(\bar{u}_+,\bar{u}_-)-\sigma(u_+,u_-)}^2
\leq 2(\sup\abs{\partial_v \sigma(\xi)})^2\Big(\abs{\bar{u}_+-u_+)}^2+\abs{\bar{u}_--u_-}^2\Big).
\end{align}

From \eqref{entropy_dissipation_room_scalar_in_use} and \eqref{control_on_difference_sigmas_scalar_2}, we receive
\begin{align}\label{entropy_dissipation_room_scalar_in_use_with_differenc_sigmas}
q(u_+;\bar{u}_+)-q(u_-;\bar{u}_-)-\dot{h}(t)(\eta(u_+|\bar{u}_+)-\eta(u_-|\bar{u}_-)) \leq -\frac{c}{2(\sup\abs{\partial_v \sigma(\xi)})^2}\abs{\sigma(\bar{u}_+,\bar{u}_-)-\sigma(u_+,u_-)}^2.
\end{align}

\uline{Step 3}
\hfill

Define
\begin{align}\label{Z_def_main_engine}
A(u|\bar{u})\coloneqq A(u)-A(\bar{u})- A'(\bar{u})(u-\bar{u}),\\
X(t)\coloneqq s(t)-h(t).\label{X_def_calc_main_engine}
\end{align}

Note that $A(u|\bar{u})$ is locally quadratic in $u-\bar{u}$.

Then from \Cref{global_entropy_dissipation_rate_systems},  we have for almost every $a_*,b_*\in[0,T)$ verifying $a_* < b_*$,
\begin{equation}
\begin{aligned}\label{nonlocal_dissipation_formula_main_engine}
&\int\limits_{-\infty}^{\infty}\eta(u(x,b_*)|\bar{u}(x+X(b_*),b_*))\,dx-\int\limits_{-\infty}^{\infty}\eta(u(x,a_*)|\bar{u}(x+X(a_*),a_*))\,dx
\\
&\hspace{.7in}\leq \int\limits_{a_*}^{b_*}q(u(h(t)+,t);\bar{u}(s(t)+,t))-q(u(h(t)-,t);\bar{u}(s(t)-,t))
\\
&\hspace{1.2in}-\dot{h}(t)\big(\eta(u(h(t)+,t)|\bar{u}(s(t)+,t))-\eta(u(h(t)-,t)|\bar{u}(s(t)-,t))\big)\,dt
\\
&\hspace{1.2in}-\int\limits_{a_*}^{b_*} \int\limits_{-\infty}^{\infty}\Bigg[\Bigg(\partial_x \bigg|_{(x+X(t),t)}\hspace{-.45in} \bar{u}(x,t)\Bigg)\eta''(\bar{u}(x+X(t),t)) A(u(x,t)|\bar{u}(x+X(t),t))
\\
&\hspace{1.2in}+\Bigg(2\partial_x\bigg|_{(x+X(t),t)}\hspace{-.45in}\bar{u}(x,t)\dot{X}(t)\Bigg)\eta''(\bar{u}(x+X(t),t))[u(x,t)-\bar{u}(x+X(t),t)]
\\
&\hspace{1.2in}-\eta'(u(x,t)|\bar{u}(x+X(t),t))G(u(\cdot,t))(x)
\\
&+
\Bigg(G(\bar{u}(\cdot,t))(x+X(t))-G(u(\cdot,t))(x)\Bigg)\eta''(\bar{u}(x+X(t),t))[u(x,t)-\bar{u}(x+X(t),t)]\Bigg]\,dxdt.
\end{aligned}
\end{equation}

Recall that because $\bar{u}$ is in the form \eqref{form_for_burgers_hilbert_paper}, we can write $\bar{u}(x,t)=\phi(x-s(t))+w(x-s(t),t)$, for $\phi$ and $w$ as in \eqref{form_for_burgers_hilbert_paper}. Then the term
\begin{align}
\int\limits_{a_*}^{b_*} \int\limits_{-\infty}^{\infty}\Bigg(\partial_x \bigg|_{(x+X(t),t)}\hspace{-.45in} \bar{u}(x,t)\Bigg)\eta''(\bar{u}(x+X(t),t)) A(u(x,t)|\bar{u}(x+X(t),t))\,dxdt
\end{align}
from \eqref{nonlocal_dissipation_formula_main_engine} becomes
\begin{align*}
&\int\limits_{a_*}^{b_*} \int\limits_{-\infty}^{\infty}\Bigg(\partial_x \bigg|_{(x+X(t),t)}\hspace{-.45in} \phi(x)+w(x-s(t),t)\Bigg)\eta''(\bar{u}(x+X(t),t)) A(u(x,t)|\bar{u}(x+X(t),t))\,dxdt\\
&=
\\&\int\limits_{a_*}^{b_*}\hspace{-.05in} \int\limits_{-\infty}^{\infty}\hspace{-.05in}\frac{2}{\pi}\Big(1+\ln\abs{x+X(t)}\Big)\sgn(x+X(t))m(x+X(t))\eta''(\bar{u}(x+X(t),t)) A(u(x,t)|\bar{u}(x+X(t),t))\,dxdt\\
&+\int\limits_{a_*}^{b_*} \int\limits_{-\infty}^{\infty} \frac{2}{\pi}\abs{x+X(t)}\ln\abs{x+X(t)}m'(x+X(t))\eta''(\bar{u}(x+X(t),t)) A(u(x,t)|\bar{u}(x+X(t),t))\,dxdt\\
&+\int\limits_{a_*}^{b_*} \int\limits_{-\infty}^{\infty}\Bigg(\partial_x \bigg|_{(x+X(t),t)}\hspace{-.45in} w(x-s(t),t)\Bigg)\eta''(\bar{u}(x+X(t),t)) A(u(x,t)|\bar{u}(x+X(t),t))\,dxdt.
\end{align*}

To handle the term, 
\begin{align}\label{need_handle_this_term}
\int\limits_{a_*}^{b_*} \int\limits_{-\infty}^{\infty}\frac{2}{\pi}\ln\abs{x+X(t)}\sgn(x+X(t))m(x+X(t))\eta''(\bar{u}(x+X(t),t)) A(u(x,t)|\bar{u}(x+X(t),t))\,dxdt,
\end{align}
we do the following:
\begin{align}
&\int\limits_{a_*}^{b_*} \int\limits_{-\infty}^{\infty}\frac{2}{\pi}\ln\abs{x+X(t)}\sgn(x+X(t))m(x+X(t))\eta''(\bar{u}(x+X(t),t)) A(u(x,t)|\bar{u}(x+X(t),t))\,dxdt\\
&=
\int\limits_{a_*}^{b_*} \hspace{-.03in}\int\limits_{B_\epsilon(-X(t))}\hspace{-.2in}\overbracket[.5pt][7pt]{\frac{2}{\pi}\ln\abs{x+X(t)}\sgn(x+X(t))m(x+X(t))}^{L^1(\mathbb{R})}\overbracket[.5pt][7pt]{\eta''(\bar{u}(x+X(t),t)) A(u(x,t)|\bar{u}(x+X(t),t))}^{L^\infty(\mathbb{R})}\,dxdt\\
&+\int\limits_{a_*}^{b_*} \hspace{-.03in}\int\limits_{(B_\epsilon(-X(t)))^{c}}\hspace{-.27in}\overbracket[.5pt][7pt]{\frac{2}{\pi}\ln\abs{x+X(t)}\sgn(x+X(t))m(x+X(t))\eta''(\bar{u}(x+X(t),t))}^{L^\infty(\mathbb{R})}\overbracket[.5pt][7pt]{A(u(x,t)|\bar{u}(x+X(t),t))}^{L^1(\mathbb{R})}\,dxdt
\end{align}
for $0<\epsilon$, and estimate these two terms by using the indicted H\"older dualities (in the $x$ coordinate). Recall the support of $m$ is contained in $[-2,2]$.

Note that for $\epsilon>0$,
\begin{align}
\norm{\ln\abs{\cdot+X(t)}}_{L^1(B_\epsilon(-X(t)))}\leq 6\epsilon+2\epsilon\abs{\ln(\epsilon)},
\end{align}
and for $0<\epsilon<2$
\begin{align}
\norm{\ln\abs{\cdot+X(t)}}_{L^\infty(\epsilon-X(t),2-X(t))}\leq \ln(2)+\abs{\ln(\epsilon)}.
\end{align}

Then, from these estimates, we get 
\begin{equation}
\begin{aligned}\label{playing_with_epsilon}
\abs{\int\limits_{a_*}^{b_*} \int\limits_{-\infty}^{\infty}\Bigg[\frac{2}{\pi}\ln\abs{x+X(t)}\sgn(x+X(t))m(x+X(t))\eta''(\bar{u}(x+X(t),t)) A(u(x,t)|\bar{u}(x+X(t),t))\Bigg]\,dxdt}\\
\leq C\int\limits_{a_*}^{b_*} \Bigg[\Big(\epsilon+\epsilon\abs{\ln(\epsilon)}\Big) \norm{A(u|\bar{u})}_{L^\infty(\mathbb{R})}+\Big(\ln(2)+\abs{\ln(\epsilon)}\Big)\norm{A(u|\bar{u})}_{L^1(\mathbb{R})}\Bigg]\,dt.
\end{aligned}
\end{equation}

Besides just the term \eqref{need_handle_this_term}, we then want to estimate from above  the rest of the terms in \eqref{nonlocal_dissipation_formula_main_engine}:
\begin{equation}
\begin{aligned}\label{big_term_estimate_above_scalar}
&\hspace{-.7in}\int\limits_{-\infty}^{\infty}\Bigg|\overbracket[.5pt][7pt]{\Bigg(\sgn(x+X(t))m(x+X(t))+\abs{x+X(t)}\ln\abs{x+X(t)}m'(x+X(t))+\Bigg(\partial_x \bigg|_{(x+X(t),t)}\hspace{-.45in} w(x-s(t),t)\Bigg)\Bigg)\eta''(\bar{u})}^{L^\infty(\mathbb{R})}\overbracket[.5pt][7pt]{A(u|\bar{u})}^{L^1(\mathbb{R})}
\\
&\hspace{.6in}+\dot{X}(t)\overbracket[.5pt][7pt]{\Bigg(2\partial_x\bigg|_{(x+X(t),t)}\hspace{-.45in}\bar{u}(x,t)\Bigg)}^{L^2(\mathbb{R})}\overbracket[.5pt][7pt]{\eta''(\bar{u}(x+X(t),t))}^{L^\infty(\mathbb{R})}\overbracket[.5pt][7pt]{[u(x,t)-\bar{u}(x+X(t),t)]}^{L^2(\mathbb{R})}
\\
&\hspace{.6in}-\overbracket[.5pt][7pt]{\eta'(u(x,t)|\bar{u}(x+X(t),t))}^{L^1(\mathbb{R})}\overbracket[.5pt][7pt]{G(u(\cdot,t))(x)}^{L^\infty(\mathbb{R})}
\\
&\hspace{.8in}+
\overbracket[.5pt][7pt]{\Bigg(G(\bar{u}(\cdot,t))(x+X(t))-G(u(\cdot,t))(x)\Bigg)}^{L^2(\mathbb{R})}\overbracket[.5pt][7pt]{\eta''(\bar{u}(x+X(t),t))}^{L^\infty(\mathbb{R})}\overbracket[.5pt][7pt]{[u(x,t)-\bar{u}(x+X(t),t)]}^{L^2(\mathbb{R})}\Bigg|\,dx.
\end{aligned}
\end{equation}

To estimate \eqref{big_term_estimate_above_scalar}, we use the H\"older dualities indicated.

Note that to handle the term
\begin{align}
\overbracket[.5pt][7pt]{\eta'(u(x,t)|\bar{u}(x+X(t),t))}^{L^1(\mathbb{R})}\overbracket[.5pt][7pt]{G(u(\cdot,t))(x)}^{L^\infty(\mathbb{R})}
\end{align}
seen in \eqref{big_term_estimate_above_scalar}, we use \eqref{requirements_G_piecewise_scalar}. Note that if $\eta(u)=\alpha u^2+\beta u+\gamma$, then $\eta'(a|b)\equiv0$ for all $a,b$ so we do not require that $G$ be bounded from $L^\infty(\mathbb{R})\to L^\infty(\mathbb{R})$.

Further, due to $G$ being translation invariant and Lipschitz continuous from $L^2(\mathbb{R})\to L^2(\mathbb{R})$,
\begin{align}\label{control_on_difference_of_Gs_scalar}
\norm{G(\bar{u}(\cdot,t))(\cdot+X(t))-G(u(\cdot,t))(\cdot)}_{L^2(\mathbb{R})}&=\norm{G(\bar{u}(\cdot+X(t),t))(\cdot)-G(u(\cdot,t))(\cdot)}_{L^2(\mathbb{R})}
\\
&\leq \mbox{Lip}[G]\norm{\bar{u}(\cdot+X(t),t)-u(\cdot,t)}_{L^2(\mathbb{R})}.
\end{align}

For $\norm{\partial_x\bar{u}(\cdot+X(t),t)}_{L^2(\mathbb{R})}\norm{\eta''(\bar{u})}_{L^\infty}\neq0$ we have, due to `Young's inequality with $\epsilon$,'
\begin{equation}
\begin{aligned}\label{youngs_inequality_part1_scalar}
&\abs{\dot{X}(t)}\norm{u(\cdot,t)-\bar{u}(\cdot+X(t),t)}_{L^2(\mathbb{R})} 
\leq 
\\
&\hspace{.7in}\frac{c}{8(\sup\abs{\partial_v \sigma(\xi)})^2\norm{\partial_x\bar{u}(\cdot+X(t),t)}_{L^2(\mathbb{R})}\norm{\eta''(\bar{u})}_{L^\infty}}(\dot{X}(t))^2
\\
&\hspace{1in}+\frac{2(\sup\abs{\partial_v \sigma(\xi)})^2\norm{\partial_x\bar{u}(\cdot+X(t),t)}_{L^2(\mathbb{R})}\norm{\eta''(\bar{u})}_{L^\infty}}{c} \norm{u(\cdot,t)-\bar{u}(\cdot+X(t),t)}_{L^2(\mathbb{R})}^2,
\end{aligned}
\end{equation}
where $c$ is from the right hand side of \eqref{entropy_dissipation_room_scalar_in_use}. Continuing, we get from \eqref{youngs_inequality_part1_scalar},
\begin{equation}
\begin{aligned}\label{youngs_inequality_scalar}
&2\abs{\dot{X}(t)}\norm{\partial_x\bar{u}(\cdot+X(t),t)}_{L^2(\mathbb{R})}\norm{\eta''(\bar{u})}_{L^\infty}\norm{u(\cdot,t)-\bar{u}(\cdot+X(t),t)}_{L^2(\mathbb{R})}
\\
&\hspace{.3in}\leq\frac{c}{4(\sup\abs{\partial_v \sigma(\xi)})^2}(\dot{X}(t))^2
\\
&\hspace{.3in}+\frac{4(\sup\abs{\partial_v \sigma(\xi)})^2\norm{\partial_x\bar{u}(\cdot+X(t),t)}_{L^2(\mathbb{R})}^2\norm{\eta''(\bar{u})}_{L^\infty}^2}{c} \norm{u(\cdot,t)-\bar{u}(\cdot+X(t),t)}_{L^2(\mathbb{R})}^2.
\end{aligned}
\end{equation}
For any $t$ such that $\norm{\partial_x\bar{u}(\cdot+X(t),t)}_{L^2(\mathbb{R})}\norm{\eta''(\bar{u})}_{L^\infty}=0$, we don't have to estimate the term
\begin{align}
\dot{X}(t)\Bigg(2\partial_x\bigg|_{(x+X(t),t)}\hspace{-.45in}\bar{u}(x,t)\Bigg)\eta''(\bar{u}(x+X(t),t))[u(x,t)-\bar{u}(x+X(t),t)].
\end{align}

Note that by the Rankine-Hugoniot relation, $\dot{s}(t)=\sigma(\bar{u}_+,\bar{u}_-)$. And, due to \eqref{ODE_scalar_shift_h}, \eqref{fact3_scalar}, \eqref{fact4_scalar}, and \eqref{RH_relation_2}, we have $\dot{h}(t)=\sigma(u_+,u_-)$. Thus,
\begin{align}\label{rewrite_dot_X_scalar}
\dot{X}(t)=\dot{s}(t)-\dot{h}(t)=\sigma(\bar{u}_+,\bar{u}_-)-\sigma(u_+,u_-).
\end{align}

We then combine \eqref{entropy_dissipation_room_scalar_in_use_with_differenc_sigmas}, \eqref{youngs_inequality_scalar}, and \eqref{rewrite_dot_X_scalar} to get that
\begin{equation}
\begin{aligned}\label{control_on_bunch_of_stuff_scalar}
&q(u(h(t)+,t);\bar{u}(s(t)+,t))-q(u(h(t)-,t);\bar{u}(s(t)-,t))
\\
&\hspace{.3in}-\dot{h}(t)\big(\eta(u(h(t)+,t)|\bar{u}(s(t)+,t))-\eta(u(h(t)-,t)|\bar{u}(s(t)-,t))\big)
\\
&\hspace{.3in}+
\frac{c}{4(\sup\abs{\partial_v \sigma(\xi)})^2}(\dot{X}(t))^2
\\
&\hspace{.3in}+
\frac{4(\sup\abs{\partial_v \sigma(\xi)})^2\norm{\partial_x\bar{u}(\cdot+X(t),t)}_{L^2(\mathbb{R})}^2\norm{\eta''(\bar{u})}_{L^\infty}^2}{c} \norm{u(\cdot,t)-\bar{u}(\cdot+X(t),t)}_{L^2(\mathbb{R})}^2 
\\
&\hspace{.5in}\leq \frac{4(\sup\abs{\partial_v \sigma(\xi)})^2\norm{\partial_x\bar{u}(\cdot+X(t),t)}_{L^2(\mathbb{R})}^2\norm{\eta''(\bar{u})}_{L^\infty}^2}{c} \norm{u(\cdot,t)-\bar{u}(\cdot+X(t),t)}_{L^2(\mathbb{R})}^2
\\
&\hspace{.5in}-\frac{c}{4(\sup\abs{\partial_v \sigma(\xi)})^2}(\dot{X}(t))^2.
\end{aligned}
\end{equation}

Recall that $A(c|d)$ and $\eta(c|d)$ are locally quadratic in $c-d$. Similarly, remark that $\eta'(c|d)$ is locally quadratic in $c-d$ by virtue of $\eta\in C^3(\mathbb{R})$. Moreover, due to the \emph{strict} convexity of $\eta$, for $c$ and $d$ in a fixed compact set, there exists $c_*>0$ such that $\eta(c|d)\geq c_*(c-d)^2$. Then, from \eqref{nonlocal_dissipation_formula_main_engine}, \eqref{big_term_estimate_above_scalar}, \eqref{playing_with_epsilon}, \eqref{control_on_difference_of_Gs_scalar}, and \eqref{control_on_bunch_of_stuff_scalar}, we find

\begin{equation}
\begin{aligned}\label{right_before_Gronwall_scalar}
&\hspace{-.4in}{{C^{*}}}\int\limits_{a_*}^{b_*}\Bigg[\Big(\abs{\ln(\epsilon)}+\norm{\partial_ x\bar{u}(\cdot,t)}_{L^2(\mathbb{R})}^2+\ln(2)\Big)\int\limits_{\mathbb{R}}\abs{u(x,t)-\bar{u}(x+X(t),t)}^2\,dx+\Big(\epsilon+\epsilon\abs{\ln(\epsilon)}\Big)\norm{A(u|\bar{u})}_{L^\infty(\mathbb{R})}\Bigg]\,dt
\\
&\hspace{1in}+{C^{*}}\int\limits_{\mathbb{R}}\abs{u(x,a_*)-\bar{u}(x+X(a_*),a_*)}^2\,dx-\int\limits_{a_*}^{b_*}\frac{1}{{C^{*}}}(\dot{X}(t))^2\,dt
\geq 
\\
&\hspace{3.5in}\int\limits_{\mathbb{R}}\abs{u(x,b_*)-\bar{u}(x+X(b_*),b_*)}^2\,dx.
\end{aligned}
\end{equation}

We have relabeled the constant $C$ to ${C^{*}}$, to denote that ${C^{*}}$ depends on the way that $\bar{u}$ depends on the logarithm (in particular, see \eqref{playing_with_epsilon}).

In particular, we can drop the last integral on the left hand side of \eqref{right_before_Gronwall_scalar} to get
\begin{equation}
\begin{aligned}\label{right_before_Gronwall_scalar462019}
&\hspace{-.4in}{C^{*}}\int\limits_{a_*}^{b_*}\Bigg[\Big(\abs{\ln(\epsilon)}+\norm{\partial_ x\bar{u}(\cdot,t)}_{L^2(\mathbb{R})}^2+\ln(2)\Big)\int\limits_{\mathbb{R}}\abs{u(x,t)-\bar{u}(x+X(t),t)}^2\,dx+\Big(\epsilon+\epsilon\abs{\ln(\epsilon)}\Big)\norm{A(u|\bar{u})}_{L^\infty(\mathbb{R})}\Bigg]\,dt
\\
&\hspace{1in}+{C^{*}}\int\limits_{\mathbb{R}}\abs{u(x,a_*)-\bar{u}(x+X(a_*),a_*)}^2\,dx\geq 
\\
&\hspace{3.5in}\int\limits_{\mathbb{R}}\abs{u(x,b_*)-\bar{u}(x+X(b_*),b_*)}^2\,dx.
\end{aligned}
\end{equation}

Remark that the map $[0,T)\ni t\mapsto \norm{\partial_ x\bar{u}(\cdot,t)}_{L^2(\mathbb{R})}^2$ is in $L^1([0,T))$ due to $\partial_x \bar{u} \in  L^2((\mathbb{R}\times[0,T))\setminus\{(x,t)|x=s(t)\})$ by the assumption that $\bar{u}$ is in the form \eqref{form_for_burgers_hilbert_paper}.

We apply the Gronwall inequality to \eqref{right_before_Gronwall_scalar462019}. This gives 
\begin{equation}
\begin{aligned}\label{main_global_stability_result_scalar_almost}
&\hspace{-.1in}\int\limits_{\mathbb{R}}\eta(u(x,b_*)|\bar{u}(x+X(b_*),b_*))\,dx\\
&\hspace{.5in}\leq e^{{C^{*}}\big(\abs{\ln(\epsilon)}+\ln(2)\big)(b_*-a_*)+{C^{*}}\int\limits_{a_*}^{b_*}\norm{\partial_ x\bar{u}(\cdot,t)}_{L^2(\mathbb{R})}^2\,dt}\Bigg({C^{*}}\int\limits_{\mathbb{R}}\eta(u(x,a_*)|\bar{u}(x+X(a_*),a_*))\,dx\\
&\hspace{.8in}+{C^{*}}T\Big(\epsilon+\epsilon\abs{\ln(\epsilon)}\Big)\Bigg)\\
&\mbox{Remarking that for $\epsilon>0$, $e^{\abs{\ln(\epsilon)}}\leq \epsilon+\frac{1}{\epsilon}$, we get,}\\
&\hspace{.5in}\leq e^{{C^{*}}\ln(2)(b_*-a_*)+{C^{*}}\int\limits_{a_*}^{b_*}\norm{\partial_ x\bar{u}(\cdot,t)}_{L^2(\mathbb{R})}^2\,dt}\Bigg( \epsilon^{{C^{*}}(b_*-a_*)}+\frac{1}{\epsilon^{{C^{*}}(b_*-a_*)}}\Bigg)\Bigg( {C^{*}}\int\limits_{\mathbb{R}}\eta(u(x,a_*)|\bar{u}(x+X(a_*),a_*))\,dx\\
&\hspace{2.2in}+{C^{*}}T\Big(\epsilon+\epsilon\abs{\ln(\epsilon)}\Big)\Bigg)
\end{aligned}
\end{equation}

Because $\epsilon>0$ is arbitrary, we can choose 
\begin{align}
\epsilon\coloneqq \int\limits_{\mathbb{R}}\eta(u(x,a_*)|\bar{u}(x+X(a_*),a_*))\,dx
\end{align}
in \eqref{main_global_stability_result_scalar_almost}.  This gives

\begin{equation}
\begin{aligned}\label{main_global_stability_result_scalar_almost1}
&\int\limits_{\mathbb{R}}\eta(u(x,b_*)|\bar{u}(x+X(b_*),b_*))\,dx\\
&\hspace{.2in}\leq {C^{*}}e^{{C^{*}}(b_*-a_*)+{C^{*}}\int\limits_{a_*}^{b_*}\norm{\partial_ x\bar{u}(\cdot,t)}_{L^2(\mathbb{R})}^2\,dt}\Bigg[ \Bigg(\int\limits_{\mathbb{R}}\eta(u(x,a_*)|\bar{u}(x+X(a_*),a_*))\,dx\Bigg)^{{{C^{*}}(b_*-a_*)}}
\\\hspace{.5in}&+\Bigg(\int\limits_{\mathbb{R}}\eta(u(x,a_*)|\bar{u}(x+X(a_*),a_*))\,dx\Bigg)^{-{{C^{*}}(b_*-a_*)}}\Bigg]\Bigg( \int\limits_{\mathbb{R}}\eta(u(x,a_*)|\bar{u}(x+X(a_*),a_*))\,dx\\
&\hspace{.5in}+\int\limits_{\mathbb{R}}\eta(u(x,a_*)|\bar{u}(x+X(a_*),a_*))\,dx\abs{\ln\Bigg(\int\limits_{\mathbb{R}}\eta(u(x,a_*)|\bar{u}(x+X(a_*),a_*))\,dx\Bigg)}\Bigg).
\end{aligned}
\end{equation}

We have absorbed the factor $T$ into ${C^{*}}$ (recall that $T$ is fixed). 

Note that for $\alpha\in[0,\infty)$ the map
\begin{align}
\alpha\mapsto \frac{\alpha\abs{\ln(\alpha)}}{\sqrt{\alpha}+\alpha^2}
\end{align}
is bounded. Thus, there is a constant ${C^{*}}>0$ such that
\begin{align}\label{control_on_xlnx}
\alpha\abs{\ln(\alpha)}\leq {C^{*}} (\sqrt{\alpha}+\alpha^2).
\end{align}

Also note that 
\begin{align}\label{alpha_beta_estimate}
\alpha^\beta \leq \alpha^{1/4}+\alpha^3
\end{align}
for all $\alpha\in[0,\infty)$ and $\beta\in[\frac{1}{4},3]$.

Then, for  $b_*-a_*\leq \frac{1}{{2C^{*}}}$, we get from \eqref{control_on_xlnx}, \eqref{alpha_beta_estimate}, and \eqref{main_global_stability_result_scalar_almost1}, 

\begin{align}\label{main_global_stability_result_scalar_almost2}
&\int\limits_{\mathbb{R}}\eta(u(x,b_*)|\bar{u}(x+X(b_*),b_*))\,dx\\
&\hspace{.5in}\leq {C^{*}}e^{{C^{*}}(b_*-a_*)+{C^{*}}\int\limits_{a_*}^{b_*}\norm{\partial_ x\bar{u}(\cdot,t)}_{L^2(\mathbb{R})}^2\,dt}\Bigg[ \Bigg(\int\limits_{\mathbb{R}}\eta(u(x,a_*)|\bar{u}(x+X(a_*),a_*))\,dx\Bigg)^{{\frac{1}{4}}}
\\&\hspace{.7in}+\Bigg(\int\limits_{\mathbb{R}}\eta(u(x,a_*)|\bar{u}(x+X(a_*),a_*))\,dx\Bigg)^{3}\Bigg].
\end{align}

Finally, take the approximate limits as $a_*\to a^{+}$ and $b_*\to b^{-}$. Recall the dominated convergence theorem and \Cref{left_right_ap_limits} (which in particular gives the existence of the approximate limits of the space integral of relative entropy). This gives \eqref{main_global_stability_result_scalar_unit_time}.
\end{proof}

\subsection{Proof of the main theorem (\Cref{main_stability_result_statement_scalar}): Induction }\label{induction_section}

In this section, we use induction to extend the stability result in \Cref{main_stability_result_statement_scalar_unit_time} to allow for the two times of interest $a$ and $b$ to be greater than $\frac{1}{C}$ time apart (where $\frac{1}{C}$ is from \Cref{main_stability_result_statement_scalar_unit_time}), thus proving \Cref{main_stability_result_statement_scalar}.

Throughout this proof, $C$ will denote a generic constant depending on $T$.

We prove the following claim,

\begin{claim}
There exists constants $C>0$ and $\rho,\gamma>1$  depending on $T$ such that for all $a,b\in[0,T)$, 

\begin{equation}
\begin{aligned}\label{main_global_stability_result_scalar_almost5}
&\hspace{-1in}{\rm ap}\,\lim_{t\to {b}^{-}}\int\limits_{\mathbb{R}}\eta(u(x,t)|\bar{u}(x,t))\,dx\\
\leq
C \Bigg[\Bigg(&{\rm ap}\,\lim_{t\to {a}^{+}}\int\limits_{\mathbb{R}}\eta(u(x,t)|\bar{u}(x,t))\,dx\Bigg)^{\frac{1}{\gamma}}\\+\Bigg(&{\rm ap}\,\lim_{t\to {a}^{+}}\int\limits_{\mathbb{R}}\eta(u(x,t)|\bar{u}(x,t))\,dx\Bigg)^{\rho}\Bigg] e^{C (b-a)+
C\int\limits_a^{b}\norm{\partial_ x\bar{u}(\cdot,t)}_{L^2(\mathbb{R})}^2\,dt}.
\end{aligned}
\end{equation}
\end{claim}

\vspace{.2in}

Let us comment on how this claim implies the main theorem (\Cref{main_stability_result_statement_scalar}): By \Cref{left_right_ap_limits}, 
\begin{align}
\label{right_limit_at_zero1}
\int\limits_{-\infty}^{\infty}\eta (u^0(x)|\bar{u}^0(x+X(0)))\,dx\geq {\rm ap}\,\lim_{t\to {0}^{+}}\int\limits_{-\infty}^{\infty}\eta (u(x,t)|\bar{u}(x+X(t),t))\,dx.
\end{align}
Further, by convexity of $\eta$,
\begin{align}
\label{from_convexity_eta_implies_theorem}
\int\limits_{-\infty}^{\infty}\eta (u(x,b)|\bar{u}(x+X(b),b))\,dx\leq {\rm ap}\,\lim_{t\to {b}^{-}}\int\limits_{-\infty}^{\infty}\eta (u(x,t)|\bar{u}(x+X(t),t))\,dx.
\end{align}


Then \eqref{main_global_stability_result_scalar_almost5}, \eqref{right_limit_at_zero1} and \eqref{from_convexity_eta_implies_theorem} imply  \eqref{main_global_stability_result_scalar_theorem_version} in the main theorem (\Cref{main_stability_result_statement_scalar}). Recall also that $X(0)=0$.

To get \eqref{L2_control_on_shift_theorem} in the main theorem, note that from \eqref{right_before_Gronwall_scalar} we get
\begin{equation}
\begin{aligned}\label{right_before_Gronwall_scalar462019_where_it_comes}
&\hspace{-.4in}C\int\limits_{a_*}^{b_*}\Bigg[\Big(\abs{\ln(\epsilon)}+\norm{\partial_ x\bar{u}(\cdot,t)}_{L^2(\mathbb{R})}^2+\ln(2)\Big)\int\limits_{\mathbb{R}}\abs{u(x,t)-\bar{u}(x+X(t),t)}^2\,dx+\Big(\epsilon+\epsilon\abs{\ln(\epsilon)}\Big)\norm{A(u|\bar{u})}_{L^\infty(\mathbb{R})}\Bigg]\,dt
\\
&\hspace{1in}+C\int\limits_{\mathbb{R}}\abs{u(x,a_*)-\bar{u}(x+X(a_*),a_*)}^2\,dx
\geq 
\int\limits_{a_*}^{b_*}\frac{1}{C}(\dot{X}(t))^2\,dt.
\end{aligned}
\end{equation}

Then, for $\epsilon$ we choose 
\begin{align*}
\epsilon\coloneqq \int\limits_{\mathbb{R}}\eta(u(x,0)|\bar{u}(x,0))\,dx.
\end{align*}

And then we bootstrap, and use \eqref{main_global_stability_result_scalar_theorem_version} to estimate the term
\begin{align*}
\int\limits_{\mathbb{R}}\abs{u(x,t)-\bar{u}(x+X(t),t)}^2\,dx
\end{align*}
in \eqref{right_before_Gronwall_scalar462019_where_it_comes}. We then use estimates similar to \eqref{control_on_xlnx} and \eqref{alpha_beta_estimate}.

We now prove the claim.

\begin{claimproof}
We prove the claim by induction. We show that  for each $n\in\mathbb{N}$, if $S\in[\frac{n-1}{C},\frac{n}{C}]$, then there exists $C>0$ and $\rho,\gamma>1$ both depending on $n$ such that \eqref{main_global_stability_result_scalar_almost5} holds for all  $a,b\in[0,T)$ with $b-a\in(0,S)$.

\vspace{.07in}

\uline{Base case}
The base case follows directly from \Cref{main_stability_result_statement_scalar_unit_time}.

\vspace{.07in}

\uline{Induction step}

\vspace{.07in}

Assume that for a fixed $n\in\mathbb{N}$, if $S\leq\frac{n}{C}$, then there exists $C>0$ and $\rho,\gamma>1$ both depending on $n$ such that \eqref{main_global_stability_result_scalar_almost5} holds for all $a,b\in[0,T)$ with $b-a\in(0,S)$. 

We now show that if $S\in[\frac{n}{C},\frac{n+1}{C}]$, then there exists $C>0$ and $\rho,\gamma>1$ both depending on $n$ such that \eqref{main_global_stability_result_scalar_almost5} holds for all $a,b\in[0,T)$ with $b-a\in(0,S)$. 

If $b-a<\frac{n}{C}$, then by the induction hypothesis we are done. Thus, assume $b-a\in[\frac{n}{C},\frac{n+1}{C})$. Then, by the induction hypothesis,

\begin{equation}
\begin{aligned}\label{main_global_stability_result_scalar_almost6}
&\hspace{-1in}{\rm ap}\,\lim_{t\to {b}^{-}}\int\limits_{\mathbb{R}}\eta(u(x,t)|\bar{u}(x,t))\,dx\\
\leq
C \Bigg[\Bigg(&{\rm ap}\,\lim_{t\to {(b-n)}^{+}}\int\limits_{\mathbb{R}}\eta(u(x,t)|\bar{u}(x,t))\,dx\Bigg)^{\frac{1}{\gamma}}
\\
\hspace{1.5in}+\Bigg(&{\rm ap}\,\lim_{t\to {(b-n)}^{+}}\int\limits_{\mathbb{R}}\eta(u(x,t)|\bar{u}(x,t))\,dx\Bigg)^{\rho}\Bigg]  e^{C n+
C\int\limits_{(b-n)}^{b}\norm{\partial_ x\bar{u}(\cdot,t)}_{L^2(\mathbb{R})}^2\,dt}.
\end{aligned}
\end{equation}

Similarly, from the induction hypothesis, we have that 

\begin{equation}
\begin{aligned}\label{main_global_stability_result_scalar_almost7}
&\hspace{-1in}{\rm ap}\,\lim_{t\to {(b-n)}^{-}}\int\limits_{\mathbb{R}}\eta(u(x,t)|\bar{u}(x,t))\,dx\\
\leq
C \Bigg[\Bigg(&{\rm ap}\,\lim_{t\to {a}^{+}}\int\limits_{\mathbb{R}}\eta(u(x,t)|\bar{u}(x,t))\,dx\Bigg)^{\frac{1}{\gamma}}
\\
\hspace{1.5in}+\Bigg(&{\rm ap}\,\lim_{t\to {a}^{+}}\int\limits_{\mathbb{R}}\eta(u(x,t)|\bar{u}(x,t))\,dx\Bigg)^{\rho}\Bigg] e^{C (b-n-a)+
C\int\limits_{a}^{b-n}\norm{\partial_ x\bar{u}(\cdot,t)}_{L^2(\mathbb{R})}^2\,dt}.
\end{aligned}
\end{equation}

Remark that for any $\alpha>0$, 
\begin{align}\label{double_root}
(\alpha^{1/\gamma}+\alpha^{\rho})^{\rho} \leq 2^{\rho} (\alpha^{\rho/\gamma} +\alpha^{\rho^2} ).
\end{align}
This is derived by considering when $\alpha<1$ and when $\alpha>1$.

Similarly, for any $\alpha>0$, 
\begin{align}\label{double_root1}
(\alpha^{1/\gamma}+\alpha^{\rho})^{1/\gamma} \leq 2^{1/\gamma} (\alpha^{\frac{1}{\gamma^2}} +\alpha^{\rho/\gamma}).
\end{align}
Combining \eqref{double_root} and \eqref{double_root1}, we get
\begin{align}\label{double_root2}
(\alpha^{1/\gamma}+\alpha^{\rho})^{\rho}+(\alpha^{1/\gamma}+\alpha^{\rho})^{1/\gamma}
\leq
2^{\rho+2}(\alpha^{\frac{1}{\gamma^2}}+\alpha^{\rho^2}).
\end{align}

This is similarly derived by considering when $\alpha<1$ and when $\alpha>1$.

We combine \eqref{main_global_stability_result_scalar_almost7} and \eqref{double_root2} to get
\begin{align}\label{induction_part1}
&\Bigg({\rm ap}\,\lim_{t\to {(b-n)}^{+}}\int\limits_{\mathbb{R}}\eta(u(x,t)|\bar{u}(x,t))\,dx\Bigg)^{\frac{1}{\gamma}}+\Bigg({\rm ap}\,\lim_{t\to {(b-n)}^{+}}\int\limits_{\mathbb{R}}\eta(u(x,t)|\bar{u}(x,t))\,dx\Bigg)^{\rho}\\
&\hspace{.5in}\leq
C \Bigg[\Bigg({\rm ap}\,\lim_{t\to {a}^{+}}\int\limits_{\mathbb{R}}\eta(u(x,t)|\bar{u}(x,t))\,dx\Bigg)^{\frac{1}{\gamma^2}}
\\&\hspace{1in}+\Bigg({\rm ap}\,\lim_{t\to {a}^{+}}\int\limits_{\mathbb{R}}\eta(u(x,t)|\bar{u}(x,t))\,dx\Bigg)^{\rho^2}\Bigg] e^{C (b-n-a)+
C\int\limits_{a}^{b-n}\norm{\partial_ x\bar{u}(\cdot,t)}_{L^2(\mathbb{R})}^2\,dt}.
\end{align}

Then, we plug \eqref{induction_part1} into \eqref{main_global_stability_result_scalar_almost6}, to get 
\begin{equation}
\begin{aligned}\label{main_global_stability_result_scalar_almost6_almost_done}
&\hspace{-1in}{\rm ap}\,\lim_{t\to {b}^{-}}\int\limits_{\mathbb{R}}\eta(u(x,t)|\bar{u}(x,t))\,dx\\
\leq
C \Bigg[\Bigg(&{\rm ap}\,\lim_{t\to {a}^{+}}\int\limits_{\mathbb{R}}\eta(u(x,t)|\bar{u}(x,t))\,dx\Bigg)^{\frac{1}{\gamma^2}}
\\
+\Bigg(&{\rm ap}\,\lim_{t\to {a}^{+}}\int\limits_{\mathbb{R}}\eta(u(x,t)|\bar{u}(x,t))\,dx\Bigg)^{\rho^2}\Bigg]  e^{C (a-b)+
C\int\limits_{a}^{b}\norm{\partial_ x\bar{u}(\cdot,t)}_{L^2(\mathbb{R})}^2\,dt}.
\end{aligned}
\end{equation}

This shows that the claim holds for $S\in[\frac{n}{C},\frac{n+1}{C}]$.

\vspace{.07in}

Thus, by the principle of induction we have shown the claim.
\end{claimproof}

\section{Appendix}

\subsection{Proof of \Cref{Filippov_existence_scalar}}\label{Filippov_existence_scalar_proof_section}

We prove \eqref{fact1_scalar}, \eqref{fact2_scalar}, and \eqref{fact3_scalar}. Our proof is based on the proof of Proposition 1 in \cite{Leger2011}, the proof of Lemma 2.2 in \cite{serre_vasseur}, and the proof of Lemma 3.5 in \cite{2017arXiv170905610K}. The properties  \eqref{fact4_scalar} and  \eqref{fact5_scalar} are not proved here; see Lemma 6 in \cite{Leger2011} and the proofs in the appendix in \cite{Leger2011}.

Define

\begin{align}
v_n(x,t)\coloneqq \int\limits_0^1 V\bigg(u(x+\frac{y}{n},t),t\bigg)\,dy.
\end{align}

Let $h_{n}$ be the solution to the ODE:
\begin{align}
  \begin{cases}\label{n_ode_scalar}
   \dot{h}_n(t)=v_n(h_n(t),t),\mbox{ for }t>0\\
   h_n(0)=x_0.
  \end{cases}
\end{align}

Due to the boundedness of $V$, $v_n$ is bounded. Further, $v_n$ is Lipschitz continuous in $x$ due to the mollification by $\frac{1}{n}$. We also have that $v_n$ is measurable in $t$. Thus \eqref{n_ode_scalar} has a unique solution $h_n$ in the sense of Carath\'eodory.

The $h_n$ are Lipschitz continuous with Lipschitz constants uniform in $n$, due to the $v_n$ being uniformly bounded in $n$ (\hspace{.07cm}$\norm{v_n}_{L^\infty}\leq \norm{V}_{L^\infty}$). By Arzel\`a--Ascoli we conclude the $h_n$ converge in $C^0(0,T)$ for any fixed $T>0$ to a Lipschitz continuous function $h$ (passing to a subsequence if necessary). Note that $\dot{h}_n$ converges in $L^\infty$ weak* to $\dot{h}$.

We define
\begin{align}
V_{\mbox{max}}(t)\coloneqq \max\{V(u_-,t),V(u_+,t)\},\\
V_{\mbox{min}}(t)\coloneqq \min\{V(u_-,t),V(u_+,t)\},
\end{align}
where $u_\pm \coloneqq u(h(t)\pm,t)$.

To show \eqref{fact3_scalar}, we will first prove that for almost every $t>0$
\begin{align}
\lim_{n\to\infty}[\dot{h}_n(t)-V_{\mbox{max}}(t)]_+=0,\label{limit_1_scalar}\\
\lim_{n\to\infty}[V_{\mbox{min}}(t)-\dot{h}_n(t)]_+=0,\label{limit_2_scalar}
\end{align}
where $[\hspace{.1cm}\cdot\hspace{.1cm}]_+\coloneqq\max(0,\cdot)$.

The proofs of \eqref{limit_1_scalar} and \eqref{limit_2_scalar} are similar. Let us show only the first one.

\begin{align}
[\dot{h}_n(t)-V_{\mbox{max}}(t)]_+\\
=\Bigg[\int\limits_0^1 V\bigg(u(h_n(t)+\frac{y}{n},t),t\bigg)\,dy-V_{\mbox{max}}(t)\Bigg]_+\\
=\Bigg[\int\limits_0^1 V\bigg(u(h_n(t)+\frac{y}{n},t),t\bigg)-V_{\mbox{max}}(t)\,dy\Bigg]_+\\
\leq\int\limits_0^1 \Big[V\bigg(u(h_n(t)+\frac{y}{n},t),t\bigg)-V_{\mbox{max}}(t)\Big]_+\,dy\\
\leq\esssup_{y\in(0,\frac{1}{n})} \Big[V\bigg(u(h_n(t)+y,t),t\bigg)-V_{\mbox{max}}(t)\Big]_+\\
\leq\esssup_{y\in(-\epsilon_n,\epsilon_n)} \Big[V\bigg(u(h(t)+y,t),t\bigg)-V_{\mbox{max}}(t)\Big]_+,\label{last_ineq_Filippov_scalar}
\end{align}
where $\epsilon_n\coloneqq \abs{h_n(t)-h(t)}+\frac{1}{n}$. Note $\epsilon_n\to0^+$.

Fix a $t\geq0$ such that $u$ has a strong trace in the sense of \Cref{strong_trace_definition}. Then by the continuity of $V$ in the $u$ slot,
\begin{align}\label{esssuplim_is_zero_scalar}
\lim_{n\to\infty}\esssup_{y\in(0,\frac{1}{n})} \Big[V\bigg(u(h(t)\pm y,t),t\bigg)-V\big(u_\pm,t\big)\Big]_+=0,
\end{align}
where $u_\pm = u(h(t)\pm,t)$.

From \eqref{esssuplim_is_zero_scalar}, we get
\begin{align}\label{esssuplim_is_zero2_scalar}
\lim_{n\to\infty}\esssup_{y\in(0,\frac{1}{n})} \Big[V\bigg(u(h(t)\pm y,t),t\bigg)-V_{\mbox{max}}(t)\Big]_+=0.
\end{align}

We can control \eqref{last_ineq_Filippov_scalar} from above by the quantity
\begin{equation}
\begin{aligned}\label{esssuplim_is_zero3_scalar}
\esssup_{y\in(-\epsilon_n,0)} \Big[V\bigg(u(h(t)+ y,t),t\bigg)-V_{\mbox{max}}(t)\Big]_++\\
\esssup_{y\in(0,\epsilon_n)} \Big[V\bigg(u(h(t)+ y,t),t\bigg)-V_{\mbox{max}}(t)\Big]_+.
\end{aligned}
\end{equation}

By \eqref{esssuplim_is_zero2_scalar}, we have that \eqref{esssuplim_is_zero3_scalar} goes to $0$ as $n\to\infty$. This proves \eqref{limit_1_scalar}.

Recall that $\dot{h}_n$ converges in $L^\infty$ weak* to $\dot{h}$. Thus,
\begin{align}
\int\limits_0^T[\dot{h}(t)-V_{\mbox{max}}(t)]_+\,dt\leq \liminf_{n\to\infty}\int\limits_0^T[\dot{h}_n(t)-V_{\mbox{max}}(t)]_+\,dt,
\end{align}
because the function $[\hspace{.1cm}\cdot\hspace{.1cm}]_+$ is convex.

We have from the dominated convergence theorem and \eqref{limit_1_scalar},
\begin{align}
\liminf_{n\to\infty}\int\limits_0^T[\dot{h}_n(t)-V_{\mbox{max}}(t)]_+\,dt=0.
\end{align}

We conclude,
\begin{align}
\int\limits_0^T[\dot{h}(t)-V_{\mbox{max}}(t)]_+\,dt=0.
\end{align}

From a similar argument,
\begin{align}
\int\limits_0^T[V_{\mbox{min}}(t)-\dot{h}(t)]_+\,dt=0.
\end{align}

This proves \eqref{fact3_scalar}.

\subsection{Proof of \Cref{left_right_ap_limits}}\label{approx_limits_appendix}
We present a proof of  \Cref{left_right_ap_limits}, based on the proof of Lemma 2.4 in \cite[p.~11]{2017arXiv170905610K}. 

Define $\Gamma:[0,T)\to\mathbb{R}$, 
\begin{align*}
&\Gamma(t)\coloneqq \int\limits_{-\infty}^{\infty}\eta (u(x,t)|\bar{u}(x+X(t),t))\,dx 
\\
&\hspace{.7in}-\int\limits_{0}^{t}q(u(h(t)+,t);\bar{u}(s(t)+,t))-q(u(h(t)-,t);\bar{u}(s(t)-,t))
\\
&\hspace{.7in}-\dot{h}(t)\big(\eta(u(h(t)+,t)|\bar{u}(s(t)+,t))-\eta(u(h(t)-,t)|\bar{u}(s(t)-,t))\big)\,dt
\\
&\hspace{.7in}+\int\limits_{0}^{t} \int\limits_{-\infty}^{\infty}\Bigg[\Bigg(\partial_x \bigg|_{(x+X(t),t)}\hspace{-.45in} \eta'(\bar{u}(x,t))\Bigg) A(u(x,t)|\bar{u}(x+X(t),t))
\\
&\hspace{.7in}+\Bigg(2\partial_x\bigg|_{(x+X(t),t)}\hspace{-.45in}\bar{u}(x,t)\dot{X}(t)\Bigg)\eta''(\bar{u}(x+X(t),t))[u(x,t)-\bar{u}(x+X(t),t)]
-
\\
&\hspace{.7in}\eta'(u(x,t)|\bar{u}(x+X(t),t))G(u(\cdot,t))(x)
\\
&\hspace{-.2in}+
\Bigg(G(\bar{u}(\cdot,t))(x+X(t))-G(u(\cdot,t))(x)\Bigg)\eta''(\bar{u}(x+X(t),t))[u(x,t)-\bar{u}(x+X(t),t)]\Bigg]\,dxdt.
\end{align*}

By \eqref{nonlocal_dissipation_formula}, $\Gamma(t)\geq\Gamma(s)$ for almost every $t$ and $s$ verifying $t<s$. We conclude that there exists a function which equals $\Gamma$ almost everywhere and is non-increasing. Thus, $\Gamma$ has approximate left and right  limits. In particular, via the dominated convergence theorem, the approximate right- and left-hand limits \eqref{ap_right_left_limits_exist_44}
exist for all $t_0\in(0,T)$ and satisfy \eqref{left_and_right_limits_order}.
Note the approximate right-hand limit also exists for $t_0=0$. Furthermore, because \eqref{nonlocal_dissipation_formula} holds for the time $a=0$, the approximate right-hand limit verifies \eqref{right_limit_at_zero} at time zero.

\bibliographystyle{plain}
\bibliography{references}
\end{document}